\documentclass[12 pt]{article}
\usepackage{amsmath,graphicx,amsfonts,amsthm,amssymb,mathrsfs,latexsym,paralist,xy}
\usepackage{tikz}
\usepackage{soul}
\usepackage{comment}

 \usepackage[normalem]{ulem}

\newtheorem{thm}{Theorem}[section]
\newtheorem{prop}[thm]{Proposition}
\newtheorem{cor}[thm]{Corollary}
\newtheorem{lemma}[thm]{Lemma}
\theoremstyle{remark}
\newtheorem{rmk}[thm]{Remark}
\newtheorem{example}[thm]{Example}
\theoremstyle{definition}
\newtheorem{defn}[thm]{Definition}
\newtheorem{proposition}[thm]{Proposition}

\newcommand{\MCE}{\operatorname{MCE}}

\newcommand{\bi}{\begin{itemize}}
	\newcommand{\ei}{\end{itemize}}
\newcommand{\be}{\begin{enumerate}}
	\newcommand{\ee}{\end{enumerate}}

\newcommand{\C}{\mathbb{C}}

\newcommand{\T}{\mathbb{T}}
\renewcommand{\H}{\mathcal{H}}

\newcommand{\R}{\mathbb{R}}
\newcommand{\N}{\mathbb{N}}
\newcommand{\Z}{\mathbb{Z}}

\providecommand{\keywords}[1]{{\textit{Keywords and phrases:}} #1}
\providecommand{\classification}[1]{{\textit{2010 Mathematics Subject Classification:}} #1}

\def\IoIIdimdots(#1/#2/#3,#4){\node at (#1,#4) {$.$};\node at (#2,#4) {$.$};\node at (#3,#4) {$.$};}
\def\IIoIIdimdots(#1,#2/#3/#4){\node at (#1,#2) {$.$};\node at (#1,#3) {$.$};\node at (#1,#4) {$.$};}

\def\IoIIIdimdots(#1/#2/#3,#4,#5){\node at (#1,#4,#5) {$.$};\node at (#2,#4,#5) {$.$};\node at (#3,#4,#5) {$.$};}
\def\IIoIIIdimdots(#1,#2/#3/#4,#5){\node at (#1,#2,#5) {$.$};\node at (#1,#3,#5) {$.$};\node at (#1,#4,#5) {$.$};}
\def\IIIoIIIdimdots(#1,#2,#3/#4/#5){\node at (#1,#2,#3) {$.$};\node at (#1,#2,#4) {$.$};\node at (#1,#2,#5) {$.$};}

\usepackage[margin=1in]{geometry}
\AtEndDocument{\bigskip{\small%
		\textsc{Carla Farsi: Department of Mathematics, University of Colorado at Boulder, Boulder, CO 80309-0395, USA.} \par
		\textit{E-mail address}: \texttt{carla.farsi@colorado.edu} \par
		\addvspace{\medskipamount}
		\textsc{Elizabeth Gillaspy : Department of Mathematics, University of Montana, 32 Campus Drive \#0864, Missoula, MT 59812-0864, USA.}\par
		\textit{E-mail address}: \texttt{elizabeth.gillaspy@mso.umt.edu}\par
		\addvspace{\medskipamount}
		\textsc{Daniel Gon\c{c}alves : Departamento de Matem\'{a}tica - UFSC - Florian\'{o}polis - SC - 88040-900, Brazil}\par
		\textit{E-mail address}: \texttt{daemig@gmail.com}\par
}}

\begin{document}
	
	\title{Irreducibility and monicity for representations of $k$-graph $C^*$-algebras}
	\author{Carla Farsi, Elizabeth Gillaspy, and Daniel Gon\c{c}alves}
	\date{\today}
	\maketitle
	
	\begin{abstract} 
	The representations of a $k$-graph $C^*$-algebra $C^*(\Lambda)$ which arise from $\Lambda$-semibranching function systems are closely linked to the dynamics of the $k$-graph $\Lambda$.  In this paper, we undertake a systematic analysis of the question of irreducibility for these representations.  We  provide a variety of necessary and sufficient conditions for irreducibility, as well as a number of examples indicating the optimality of our results.  We also explore the relationship between irreducible $\Lambda$-semibranching representations and purely atomic representations of $C^*(\Lambda)$.  
	Throughout the paper, we work in the setting of  row-finite source-free $k$-graphs; this paper constitutes the first analysis of $\Lambda$-semibranching representations at this level of generality. % We show that  much of the structure theory of 	$\Lambda$-semibranching representations  which had been established for finite $k$-graphs  (such as their connection to  monic representations) also holds in  the row-finite setting.
	\end{abstract}
	
	\classification{46L05, 46L55, 46K10}
	

	\tableofcontents
\section{Introduction}
Understanding the irreducible representations of a $C^*$-algebra  enables an analysis of its spectrum and primitive ideal space, as well as its representation theory. In addition,  longstanding open questions such as Naimark's problem \cite{Nai51} use irreducible representations to (conjecturally) determine how similar two $C^*$-algebras are.   {However, for non-type I $C^*$-algebras $A$, it is well known \cite{Glimm1, Glimm2} that the natural Borel structure on the space $\widehat A$ of irreducible representations of $A$ is not countably separated.  As many higher-rank graph $C^*$-algebras are not type I, this lack  of a ``reasonable'' parametrization of their irreducible representations  obligates us  to restrict our attention to specific subfamilies of representations when exploring the question of irreducibility.}
In this paper, we study the irreducibility of  the representations of higher-rank graph $C^*$-algebras $C^*(\Lambda)$ which arise from $\Lambda$-semibranching function systems.   

For graph $C^*$-algebras and related constructions, the representations arising from branching systems \cite{MP, DD, DD1, goncalves-li-royer-SBFS, FGKP} provide a key insight into the structure of the  $C^*$-algebra. {Intuitively, a branching system is a family of subsets of a measure space $(X, \mu)$ which reflects the structure of the graph $C^*$-algebra $C^*(\Lambda)$, so that one obtains a natural representation of $C^*(\Lambda)$ on $L^2(X, \mu)$. Indeed, this representation encodes the natural dynamics of the prefixing and coding maps on the space $\Lambda^\infty$ of infinite paths in the graph $\Lambda$: the fact that $C_0(\Lambda^\infty)$ is a subalgebra of $C^*(\Lambda)$ implies that the structure of $C_0(\Lambda^\infty)$, and in particular the dynamics of changing an infinite path $x_1 x_2 x_3 \ldots $ to $x_0 x_1 x_2 x_3 \ldots$ or to $x_2 x_3 \ldots$, must also be reflected in the branching system on $(X, \mu)$. Consequently,  the  study of representations arising from branching systems also enhances our understanding of the symbolic dynamics associated to a graph or higher-rank-graph.}

 In some settings, in fact, every representation  arises from a branching system: {the third author together with D. Royer identify in }\cite{DD, DD1}  a class of graphs for which every representation of the graph $C^*$-algebra is unitarily equivalent to a representation arising from a branching system.  
For higher-rank graphs $\Lambda$, branching systems were introduced as $\Lambda$-semibranching function systems {by the first and second authors together with S. Kang and J. Packer}  in \cite{FGKP}, where  the associated representations were used to construct wavelets and to analyze the KMS states of the higher-rank graph $C^*$-algebra $C^*(\Lambda)$.  Subsequent work by Farsi, Gillaspy, Kang, and Packer together with P. Jorgensen \cite{FGJKP-monic} showed that a large class of representations of $C^*(\Lambda)$ -- the so-called monic representations -- all arise from $\Lambda$-semibranching function systems, and these authors provide in \cite{FGJKP-atomic}  a more detailed analysis of the structure of $\Lambda$-semibranching function systems in the case when the associated measure space is atomic.
The question of when a $\Lambda$-semibranching representation is faithful has been explored in \cite{FGKP,  goncalves-li-royer-SBFS, FGJKP-SBFS}. {However, even in the previously mentioned studies, there has not been a specific emphasis on irreducible representations, 
which instead constitute the focus of this research.}  {Although the primitive ideal space of higher-rank graph $C^*$-algebras is well understood \cite{CKSS, Kang-Pask},} {the broad applicability of $\Lambda$-semibranching function systems has inspired us to explore the question of when the associated representations of $C^*(\Lambda)$ are irreducible.}

Higher-rank graphs (or $k$-graphs) are a $k$-dimensional generalization of directed graphs (1-graphs). Introduced by Kumjian and Pask in \cite{KP}, $k$-graphs provide {a framework for} combinatorial constructions of $C^*$-algebras {and shift spaces}, beyond the setting of graph and Cuntz--Krieger $C^*$-algebras, {and graph and Markov shift spaces. In fact,} in the decades since their introduction, $k$-graph $C^*$-algebras have led to advances in symbolic dynamics  \cite{skalski-zacharias, carlsen-ruiz-sims} and noncommutative geometry \cite{hajac-sims}, as well as insights into the dimension theory of $C^*$-algebras \cite{ruiz-sims-sorensen}.
Higher-rank graph $C^*$-algebras  share many of the important properties of Cuntz and Cuntz--Krieger  $C^*$-algebras (of which they are generalizations~\cite{Cuntz, Cuntz-Krieger,enomoto-watatani,kpr}), 
including Cuntz--Krieger uniqueness theorems and realizations as 
groupoid $C^*$-algebras.

{As we mentioned above} the main objective of this paper is to provide a variety of necessary and sufficient conditions for a $\Lambda$-semibranching function system to give rise to an irreducible representation of $C^*(\Lambda)$. 
 {In our work we will establish a strong link beteween the representation theory of $k$-graph $C^*$-algebras and the symbolic dynamics associated to a $k$-graph, by detecting irreducibility of a representation arising from $\Lambda$-semibranching functions systems in terms of conditions on the coding maps.} We focus on the setting of row-finite source-free $k$-graphs, whose $\Lambda$-semibranching function systems have not yet been analyzed in the literature. There are  nontrivial structural differences between the finite $k$-graph case and the row-finite one: for example, the infinite path space in the row-finite case is typically 
non-compact. To obtain our main results, we consequently need to extend a  number of results which had previously been established in the literature only for finite $k$-graphs.

{We now describe 
the content of this paper}.  In addition to reviewing the relevant background material (including definitions of higher-rank graphs and their $\Lambda$-semibranching function systems) in Section \ref{sec:background}, we also establish two new results in that section. 
Namely, Proposition~\ref{prop:restriction-gives-SBFS} shows that, given an invariant subset $E\subseteq X$ and a $\Lambda$-semibranching representation on $(X, \mu)$, we also obtain a $\Lambda$-semibranching representation on $L^2(E, \mu)$ under mild hypotheses.  Theorem \ref{thm:meas-from-eigenvect} describes a method for constructing measures on the infinite path space of $\Lambda^\infty$, ensuring that we have a supply of examples of $\Lambda$-semibranching function systems. 
Then, Section~\ref{sec:monicity} establishes that for row-finite higher-rank graphs, the  characterization of those $\Lambda$-semibranching representations which are unitarily equivalent to representations on the infinite path space $\Lambda^\infty$ is analogous to the known characterization  for finite higher-rank graphs.  {In more detail, Theorem~\ref{thm-characterization-monic-repres} and Theorem~\ref{thm:range-sets-generate-sigma-alg-in-monic} give respectively a representation-theoretic and a measure-theoretic characterization of when a $\Lambda$-semibranching representation is equivalent to one arising from $\Lambda^\infty$. Even when restricted to the setting of finite higher-rank graphs, Theorem~\ref{thm:range-sets-generate-sigma-alg-in-monic} is stronger than the existing results in the literature.
We also explore the relationship between monic representations and the periodicity of $\Lambda$ in Section~\ref{sec:monicity}, and obtain in Corollary~\ref{prop:no-entrance-implies-no-monic-repns}
a new necessary condition for the monicity of branching representations for directed graphs containing cycles without entrance.  }

Section~\ref{sec:irred} is the main contribution of this paper {and where we establish connections between irreducibility of a representation arising from $\Lambda$-semibranching function system and the dynamics associated to a $k$-graph}. Here we provide a variety of necessary (Propositions~\ref{ventosul} and \ref{prop:irred-implies}, and Theorem~\ref{prop:cofinal})
and sufficient (Theorems~\ref{thm:ergodic} and \ref{thm:ergodic-inf-path-space}, and Proposition \ref{thm:OA-irred})
conditions for a $\Lambda$-semibranching function system to give rise to an irreducible representation of $C^*(\Lambda)$.  
{Many of these necessary and sufficient conditions deal with the ergodicity of the dynamics of the $\Lambda$-semibranching function system.}  We also wish to highlight  Theorem~\ref{prop:cofinal}, which shows that only cofinal $k$-graphs admit irreducible $\Lambda$-semibranching representations.
 {Although we provide examples which indicate that our necessary conditions are not in general sufficient, when the measure space $X$ of  the $\Lambda$-semibranching function system is $\Lambda^\infty$, the situation is different:  Proposition~\ref{ventosul} and Theorem~\ref{thm:ergodic} combine to imply that a $\Lambda$-semibranching representation on $(\Lambda^\infty, \mu)$ is irreducible precisely when the coding maps are jointly ergodic with respect to $\mu$.  Motivated in part by Theorem \ref{prop:cofinal},
 we also provide  sufficient conditions in Theorems~\ref{thm:ergodic} and \ref{thm:ergodic-inf-path-space} for irreducibility of a $\Lambda$-semibranching representation arising from a proper subset $X$ of $\Lambda^\infty$. {In particular, Example~\ref{ex:not-cofinal-but-irred}  shows that, using Theorem \ref{thm:ergodic-inf-path-space}, one can  use a $\Lambda$-semibranching function system to obtain an irreducible representation of $C^*(\Lambda)$ even if $\Lambda$ is not necessarily cofinal. }

We conclude this paper in Section~\ref{sec:atomic} by studying the irreducibility of atomic $\Lambda$-semibranching representations. {Indeed, Theorem~\ref{thm:irred-repres-atom-purely-atomic} shows that any irreducible $\Lambda$-semibranching representation on an atomic measure space is purely atomic in the sense of \cite{FGJKP-atomic}, and 
 Theorem~\ref{vinhosoon} shows in particular that irreducible representations on atomic measure spaces are monic.} Finally, in Section~\ref{sec:naimark}, we present an application of $\Lambda$-semibranching representations in the context of the Naimark problem for graph algebras.

 	\subsection*{Acknowledgments}   	  
	C.F.~was  partially supported by the Simons Foundation Collaboration Grant for Mathematics \#523991.  E.G.~was partially supported by the National Science Foundation (DMS-1800749). D.G.~was partially supported by Conselho Nacional de Desenvolvimento Cient\'ifico e Tecnol\'ogico (CNPq), and by Capes-PrInt grant number 88881.310538/2018-01 - Brazil.
	C.F.~also thanks the sabbatical program at the University of Colorado/Boulder for support. The authors thank Judith Packer for helpful conversations.
	
	\section{{Preliminary} material}
	\label{sec:background}
	 In this section we  recall the definition of  higher-rank graphs and their $C^*$-algebras from  \cite{KP}, together with results on $\Lambda$-semibranching function systems and their associated representations that extend those established for finite higher-rank graphs  in \cite{FGKP} and \cite{FGJKP-monic}. 
	  
	  \subsection{Definition of higher-rank graphs }
	  
	Let $\N=\{0,1,2,\dots\}$ denote the monoid of natural numbers under addition, and let $k\in \N$ with $k\ge 1$. We write $e_1,\dots, e_k$ for the standard basis vectors of $\N^k$, where $e_i$ is the vector of $\N^k$ with $1$ in the $i$-th position and $0$ everywhere else.
	
	\begin{defn} \cite[Definition 1.1]{KP}
	\label{def-higher-rank-gr}
	A \emph{higher-rank graph} or \emph{$k$-graph} is a countable small category\footnote{Recall that a small category is one  in which  the collection of arrows is a  set.} $\Lambda$ 
	 with a degree functor $d:\Lambda\to \N^k$ satisfying the \emph{factorization property}: for any morphism $\lambda\in\Lambda$ and any $m, n \in \N^k$ such that  $d(\lambda)=m+n \in \N^k$,  there exist unique morphisms $\mu,\nu\in\Lambda$ such that $\lambda=\mu\nu$ and $d(\mu)=m$, $d(\nu)=n$. 
	\end{defn}
	When discussing $k$-graphs, we use the arrows-only picture of category theory; thus, objects in $\Lambda$ are identified with identity morphisms, and the notation $ \lambda \in \Lambda $ means $\lambda$ is a morphism in $\Lambda$.
	\begin{example}
	\label{ex:examples-of-kgraphs}
	{\color{white},}	
	
	\begin{enumerate}	
	\item 
	The higher-rank graphs with $k =1$ (the 1-graphs) correspond to the categories whose objects are the vertices of a directed graph $E$, and whose morphisms are the finite paths in $E$.  In this case,   $d(\lambda) \in \N$ is the number of edges in $\lambda$.
	\item The $k$-graph $\Omega_k$ has $\text{Obj}(\Omega_k) = \N^k$ and $\text{Mor}(\Omega_k) = \{ (m, n) \in \N^k \times \N^k: m \leq n\}$.  We have $d(m, n) = n-m$.
	\end{enumerate}
	\end{example}	
We often regard $k$-graphs as a $k$-dimensional generalization of directed graphs, so we call morphisms $\lambda\in\Lambda$ \emph{paths} in $\Lambda$, and the objects (identity morphisms) are often called \emph{vertices}. For $n\in\N^k$, we write
	\begin{equation}
	\label{eq:Lambda-n}
	\Lambda^n:=\{\lambda\in\Lambda\,:\, d(\lambda)=n\}\
	\end{equation}
	With this notation, note that $\Lambda^0$ is the set of objects (vertices) of $\Lambda$, and we will call elements of $\Lambda^{e_i}$ (for any $i$) \emph{edges}.  
	We write $r,s:\Lambda\to \Lambda^0$ for the range and source maps in $\Lambda$ respectively.  For vertices $v, w \in \Lambda^0$, we define
	\[v\Lambda w:=\{\lambda\in\Lambda\,:\, r(\lambda)=v,\;s(\lambda)=w\} \quad \text{and} \quad  v\Lambda^n:= \{ \lambda \in \Lambda: r(\lambda) = v, \ d(\lambda) = n\}.\]
	
	Our focus in this paper is on  row-finite $k$-graphs with no sources.
		We say that $\Lambda$  \emph{has no sources} or \emph{is source-free} if $v\Lambda^n\ne \emptyset$ for all $v\in\Lambda^0$ and $n\in\N^k$. It is well known that this is equivalent to the condition that $v\Lambda^{e_i}\ne \emptyset$ for all $v\in \Lambda$ and all basis vectors $e_i$ of $\N^k$. 
			A $k$-graph $\Lambda $ is \emph{row-finite} if 
				
				\begin{equation}
				\label{def-row finite}
				\#\Big( v\Lambda^n \Big) < \infty,\ \forall v \in \Lambda^0,\ \forall n \in \N^k.
				\end{equation} 
	
	For $m,n\in\N^k$, we write $m\vee n$ for the coordinatewise maximum of $m$ and $n$. Given  $\lambda,\eta\in \Lambda$, we write
	\begin{equation}\label{eq:lambda_min}
	\Lambda^{\operatorname{min}}(\lambda,\eta):=\{(\alpha,\beta)\in\Lambda\times\Lambda\,:\, \lambda\alpha=\eta\beta,\; d(\lambda\alpha)=d(\lambda)\vee d(\eta)\}.
	\end{equation}
	If $k=1$, then $\Lambda^{\operatorname{min}}(\lambda, \eta)$ will have at most one element; this need not be true  if $k > 1$.
	
	Related to the definition of $\Lambda^{\min}(\lambda,\nu)$ is 
	\[ \text{MCE}(\lambda, \nu) = \{ \xi \in \Lambda: \ \exists (\alpha, \beta) \in \Lambda^{\min}(\lambda, \nu) \text{ such that } \xi = \lambda \alpha = \nu \beta\}.\]
	Observe that if $r(\lambda) \not= r(\nu)$ then $\text{MCE}(\lambda, \nu) = \emptyset =  \Lambda^{\min}(\lambda, \nu)$; if $r(\lambda) = r(\nu)$ then $r(\xi) = r(\lambda)$ for all $\xi \in \text{MCE}(\lambda, \nu)
	$.

	\begin{defn}[\cite{KP} Definitions 2.1]
	\label{def:infinite-path}
	Let $\Lambda$ be a $k$-graph. An \emph{infinite path} in $\Lambda$ is a $k$-graph morphism (degree-preserving functor) $x:\Omega_k\to \Lambda$, and we write $\Lambda^\infty$ for the set of infinite paths in $\Lambda$. 
	Since $\Omega_k$ has a terminal object (namely $0 \in\N^k$) but no initial object, we think of our infinite paths as having a range $r(x) : = x(0)$ but no source.
	For each $m\in \N^k$, we have a shift map $\sigma^m:\Lambda^\infty \to \Lambda^\infty$ given by
	\begin{equation}\label{eq:shift-map}
	\sigma^m(x)(p,q)=x(p+m,q+m)
	\end{equation}
	for $x\in\Lambda^\infty$ and $(p,q)\in\Omega_k$. 

	It is well-known that the collection of cylinder sets 
	\[
	Z(\lambda)=\{x\in\Lambda^\infty\,:\, x(0,d(\lambda))=\lambda\},
	\]
	for $\lambda \in \Lambda$, form a compact open basis for a locally compact Hausdorff topology on $\Lambda^\infty$ if  $\Lambda$ is row-finite: see Section 2 of \cite{KP}. {The cylinder sets also generate the standard Borel structure $\mathcal{B}_o(\Lambda^\infty)$ on $\Lambda^\infty$.
In particular, if we enumerate the vertices, say $\{v_i\}_{i\in \N}$ of $\Lambda,$  recall that the $\sigma$-algebra $\mathcal{A}$ of the disjoint union
$\Lambda^\infty= \bigsqcup_{n \in \N} v_n \Lambda^\infty$ is defined to be
\[
\mathcal{A} = \Big\{  A\subseteq \Lambda^\infty\hbox{ such that } A\cap v_j\Lambda^\infty \hbox{ Borel, }\forall j\in \N\Big\}.
\]}
	
	We also have a partially defined ``prefixing map'' $\sigma_\lambda: Z(s(\lambda)) \to Z(\lambda)$ for each $\lambda \in \Lambda$:
	\[ \sigma_\lambda(x) = \lambda x = \left[ (p, q) \mapsto \begin{cases} \lambda(p, q), & q \leq d(\lambda) \\
	x(p-d(\lambda), q-d(\lambda)), & p \geq d(\lambda) \\
	\lambda (p, d(\lambda)) x(0, q-d(\lambda)), & p < d(\lambda) < q
	\end{cases} \right]
	\]
	\end{defn}

	\begin{defn}
	\label{def:cofinal}
The {\em orbit} of an infinite path $x$ is 
\begin{equation}
\text{Orbit}(x) = \{ y \in \Lambda^\infty: \exists m, n \in \N^k \text{ s.t. } \sigma^n(x) = \sigma^m(y)\} = \{ \lambda\sigma^n(x): n \in \N^k, \lambda \in \Lambda\}.
\label{eq:orbit}
\end{equation}
	We say that a $k$-graph $\Lambda$ is {\em cofinal} if, and only if, for all $x \in \Lambda^\infty$ and $v \in \Lambda^0$, there exists $y\in \text{Orbit}(x)$ with $ r(y) = v$.  
	\end{defn}
	
	Now we introduce the $C^*$-algebra associated to a row-finite, source-free $k$-graph $\Lambda$. 
	
	\begin{defn} (\cite[Definition 1.5]{KP})\label{def:kgraph-algebra}
	Let $\Lambda$ be a row-finite $k$-graph with no sources. A \emph{Cuntz--Krieger $\Lambda$-family} is a collection  $\{t_\lambda:\lambda\in\Lambda\}$ of partial isometries in a $C^*$-algebra satisfying
	\begin{itemize}
	\item[(CK1)] $\{t_v\,:\, v\in\Lambda^0\}$ is a family of mutually orthogonal projections,
	\item[(CK2)] $t_\lambda t_\eta=t_{\lambda\eta}$ if $s(\lambda)=r(\eta)$,
	\item[(CK3)] $t^*_\lambda t_\lambda=t_{s(\lambda)}$ for all $\lambda\in\Lambda$,
	\item[(CK4)] for all $v\in\Lambda$ and $n\in\N^k$, we have $
	t_v=\sum_{\lambda\in v\Lambda^n} t_\lambda t^*_\lambda.$
	\end{itemize}
	The Cuntz--Krieger $C^*$-algebra $C^*(\Lambda)$ associated to $\Lambda$ is the universal $C^*$-algebra generated by a Cuntz--Krieger $\Lambda$-family, in the sense that for any Cuntz--Krieger $\Lambda$-family $\{t_\lambda: \lambda \in \Lambda\}$, there is an onto $*$-homomorphism $C^*(\Lambda) \to C^*(\{t_\lambda: \lambda \in \Lambda\})$.  We will usually write $s_\lambda$ for the generator of $C^*(\Lambda)$ corresponding to $\lambda \in \Lambda$.
	\end{defn}
	
	Since the sum of two projections is a projection iff the summands are mutually orthogonal, (CK4) implies that  $t_\lambda t_\lambda^* \perp t_\eta t_\eta^*$ if $\lambda \not= \eta$.  Also,
	conditions (CK2) - (CK4) implies that for all $\lambda, \eta\in \Lambda$, we have
	\begin{equation}\label{eq:CK4-2}
	t_\lambda^* t_\eta=\sum_{(\alpha,\beta)\in \Lambda^{\operatorname{min}}(\lambda,\eta)} t_\alpha t^*_\beta. 
	\end{equation}
	It follows that $
	C^*(\Lambda)=\overline{\operatorname{span}}\{s_\alpha s^*_\beta\,:\, \alpha,\beta\in\Lambda,\; s(\alpha)=s(\beta)\}.$

	\subsection{$\Lambda$-semibranching function systems, $\Lambda$-projective systems, and   representations}
	
	In \cite{FGKP}, separable representations of $C^*(\Lambda)$ (when $\Lambda$ is finite) were constructed by using $\Lambda$-semibranching function systems on measure spaces.  Intuitively, a $\Lambda$-semibranching function system is a way of encoding the Cuntz-Krieger relations at the measure-space level: the prefixing map $\tau_\lambda$ corresponds to the partial isometry $s_\lambda \in C^*(\Lambda)$.   The construction of a $\Lambda$-semibranching function system  from  \cite[Section 3.1]{FGKP} extends verbatim to the row-finite case; we provide the details below.
	
	\begin{defn} \cite[Definition 2.1]{MP}
\label{def:SBFS}	
	Let $(X, \mu)$ be a measure space, and let $I$ be a finite or countable set of indices.  Suppose that, for each $i\in I$, we have a measurable subset $D_i \subseteq X$, with $0 < \mu(D_i) < \infty$ for all $i$, and a measurable map $\sigma_i: D_i \to X$. The family $\{\sigma_i\}_{i\in I}$ is a {\em semibranching function system} if the following hold.
	\begin{enumerate}
	\item Writing $R_i = \sigma_i(D_i)$, we have 
	\[ \mu(X \backslash \bigcup_i R_i) = 0, \qquad \mu(R_i \cap R_j) = 0 \text{ for } i \not= j, 
\]
and $\mu(R_i) < \infty$ for all $i$.
	\item The  Radon--Nikodym derivative $\Phi_i :=  \frac{d(\mu \circ \sigma_i)}{d\mu}$ is strictly positive $\mu$-a.e.~on $D_i$.
\end{enumerate}
A measurable map $\sigma: X \to X$ is called a {\em coding map} for the family $\{\sigma_i\}_{i\in I}$ if $\sigma \circ \sigma_i = id_{D_i}$ for all $i$.
	\end{defn}
	Since $\mu (R_i) = \mu \circ \sigma_i (D_i) = \int_{D_i} \frac{d(\mu \circ \sigma_i)}{d\mu} d\mu$, and $\mu(D_i)>0$, the hypothesis that the Radon--Nikodym derivative is strictly positive  implies that $0 < \mu(R_i)$ always.

	\begin{defn}\cite[Definition~3.2]{FGKP}
	\label{def-lambda-SBFS-1}
	Let $\Lambda$ be a row-finite source-free $k$-graph and let $(X, \mu)$ be a measure space.  A \emph{$\Lambda$-semibranching function system} on $(X, \mu)$ is a collection $\{D_\lambda\}_{\lambda \in \Lambda}$ of measurable subsets of $X$, together with a family of prefixing maps $\{\tau_\lambda: D_\lambda \to X\}_{\lambda \in \Lambda}$, and a family of coding maps $\{\tau^m: X \to X\}_{m \in \N^k}$, such that
	\begin{itemize}
	\item[(a)] For each $m \in \N^k$, the family $\{\tau_\lambda: d(\lambda) = m\}$ is a semibranching function system, with coding map $\tau^m$.
	\item[(b)] If $ v \in \Lambda^0$, then  $\tau_v = id$.
	\item[(c)] Let $R_\lambda = \tau_\lambda( D_\lambda)$. For each $\lambda \in \Lambda, \nu \in s(\lambda)\Lambda$, we have $R_\nu \subseteq D_\lambda$ (up to a set of measure 0), and
	\[\tau_{\lambda} \tau_\nu = \tau_{\lambda \nu}\text{ a.e.}\]
	 (Note that this implies that up to a set of measure 0, $D_{\lambda \nu} = D_\nu$ whenever $s(\lambda) = r(\nu)$).
	\item[(d)] The coding maps satisfy $\tau^m \circ \tau^n = \tau^{m+n}$ for any $m, n \in \N^k$.  (Note that this implies that the coding maps pairwise commute.)
	\end{itemize}
	\end{defn}

	As established in \cite{FGKP} in the case of finite $k$-graphs, any $\Lambda$-semibranching function system gives rise to a representation of $C^*(\Lambda)${ via \lq prefixing' and \lq chopping off' operators that satisfy the Cuntz--Krieger relations}. 
	For the convenience of the reader, we recall the formula for these $\Lambda$-semibranching representations of $C^*(\Lambda)$.
	The following theorem is an extension  of \cite[Theorem~3.5]{FGKP} to the row-finite case (cf.~also \cite[Theorem~3.5]{goncalves-li-royer-SBFS}); the proof for finite $k$-graphs given in \cite{FGKP} extends verbatim to the row-finite setting.

	\begin{thm}\cite[Theorem~3.5]{FGKP}, \cite[Theorem~3.5]{goncalves-li-royer-SBFS}	\label{thm:SBFS-repn}
	Let $\Lambda$ be a row-finite $k$-graph with no sources and suppose that we have a $\Lambda$-semibranching function system on a  measure space $(X,\mu)$ with prefixing maps $\{\tau_\lambda: \lambda \in \Lambda\}$ and coding maps $\{\tau^m:m\in \N^k\}$. For each $\lambda\in\Lambda$, define an operator $S_\lambda$  on $L^2(X,\mu)$ by
	\[
	S_\lambda\xi(x)=\chi_{R_\lambda}(x)(\Phi_{\lambda}(\tau^{d(\lambda)}(x)))^{-1/2} \xi(\tau^{d(\lambda)}(x)).
	\]
	Then the operators $\{S_\lambda:\lambda\in\Lambda\}$ form a Cuntz--Krieger $\Lambda$-family and hence generate a representation $\pi$ of $C^*(\Lambda)$ on $L^2(X, \mu)$.
	\end{thm}

We now recall  the  definition of a $\Lambda$-projective system from \cite{FGJKP-monic}.
Roughly speaking, a $\Lambda$-projective system on $(X,\mu)$ consists of a $\Lambda$-semibranching function system plus some extra information (encoded in the functions $f_\lambda$ below).

\begin{defn}\label{def:lambda-proj-system}
Let $\Lambda$ be a row-finite $k$-graph with no sources. 
A \emph{$\Lambda$-projective system} on a measure space $(X,\mu)$ is a $\Lambda$-semibranching function system on $(X,\mu)$, with prefixing maps $\{\tau_\lambda:D_\lambda\to R_\lambda\}_{\lambda\in\Lambda}$ and coding maps $\{\tau^n:n\in\N^k\}$
together with a family of functions $\{ f_\lambda\}_{\lambda\in \Lambda} \subseteq  L^2(X,\mu)$ satisfying the following conditions:
\begin{itemize}
\item[(a)] For any   $\lambda \in \Lambda$,   we have  $ 0\not=\frac{d(\mu \circ (\tau_\lambda)^{-1})}{d\mu} = |f_\lambda |^2$; 
\item[(b)] For any $\lambda, \nu \in \Lambda$, we have 
$f_\lambda \cdot (f_\nu \circ \tau^{d(\lambda)}) = f_{\lambda \nu}.$
\end{itemize}
\end{defn}

Recall from \cite[Remark 3.3]{FGJKP-monic} that the functions $f_\lambda$ vanish outside $R_\lambda$, because the same is true for the Radon--Nikodym derivative $\frac{d(\mu \circ (\tau_\lambda)^{-1})}{d\mu}$.
 
 Condition (b) of Definition~\ref{def:lambda-proj-system} is necessary in order to associate a Cuntz--Krieger $\Lambda$-family
 to a $\Lambda$-projective system.  To be precise, we have the following Proposition, which was established for finite $k$-graphs in \cite[Proposition 3.4]{FGJKP-monic}.

  \begin{prop}
  \label{prop:lambda-proj-repn}
  Let $\Lambda$ be a row-finite, source-free $k$-graph. Suppose that a measure space $(X,\mu)$ admits a $\Lambda$-semibranching function system with prefixing maps $\{\tau_\lambda:\lambda\in\Lambda\}$ and coding maps $\{\tau^n:n\in\N^k\}$.  Suppose that $\{f_\lambda\}_{\lambda\in \Lambda}$ is a collection of functions satisfying Condition (a) of Definition \ref{def:lambda-proj-system}. Then the maps $\{\tau_\lambda\}$, $\{\tau^n\}$ and $\{f_\lambda\}_\lambda$ form a $\Lambda$-projective system on $(X, \mu)$ if and only if the operators $T_\lambda \in B(L^2(X, \mu))$ given by 
  \begin{equation}\label{eq:T-lambda}
  T_\lambda(f) = f_\lambda \cdot (f \circ \tau^{d(\lambda)})
  \end{equation}
 form a Cuntz--Krieger $\Lambda$-family with each $T_\lambda$ nonzero (and hence give a representation of $C^*(\Lambda)$).
  \end{prop}

  	\begin{proof} The proof given in \cite[Proposition 3.4]{FGJKP-monic} for finite higher-rank graphs holds verbatim for row-finite $k$-graphs.
 \end{proof}
 
We call the representation given in Equation \eqref{eq:T-lambda} a {\em $\Lambda$-projective representation}.
 
 \begin{example} 
 For any $\Lambda$-semibranching function system on $(X, \mu)$, there is a natural choice of an associated $\Lambda$-projective system; namely, for $\lambda \in \Lambda^{n}$ we define 
 \begin{equation}
 \label{eq:std-f-lambda}
 f_\lambda(x) := \Phi_{\lambda}(\tau^{n}(x))^{-1/2}  \chi_{R_\lambda}(x).\end{equation}
Condition (a) is satisfied because of the hypothesis that the Radon--Nikodym derivatives be strictly positive $\mu$-a.e.~on their domain of definition.
\label{ex:standard-SBFS-is-proj}
Since the operators $S_\lambda \in B(L^2(X, \mu))$ of Theorem \ref{thm:SBFS-repn} are given by 
\[S_\lambda(f) = f_\lambda \cdot (f \circ \tau^n),\]
and Theorem \ref{thm:SBFS-repn} establishes that $\{S_\lambda\}_{\lambda\in \Lambda}$ is a Cuntz--Krieger $\Lambda$-family, Proposition \ref{prop:lambda-proj-repn}  shows that Equation \eqref{eq:std-f-lambda}  indeed describes a $\Lambda$-projective system. 
\end{example}
 
 	\begin{rmk}
 	\label{rmk:Lambda-proj-is-mult}
If $\{T_\lambda\}_{\lambda \in \Lambda}$ is a $\Lambda$-projective representation, then one computes that
\[ T_\lambda^* f = \frac{\chi_{D_\lambda} \cdot ( f \circ \tau_\lambda)}{f_\lambda \circ \tau_\lambda}.\]
It now follows, using the fact that $\tau_\lambda \circ \tau^{d(\lambda)}|_{R_\lambda} = id$, that
 	\begin{equation}
 \label{eq:lambda-proj-is-mult}
 T_\lambda T_\lambda^* = M_{\chi_{R_\lambda}}.\end{equation}
Moreover, Example \ref{ex:standard-SBFS-is-proj} tells us that Equation \eqref{eq:lambda-proj-is-mult} also holds for any $\Lambda$-semibranching representation.
 	\end{rmk}
 	
 	The following lemma  will be used in Proposition~\ref{prop:restriction-gives-SBFS} below, as well as later in Lemma \ref{lem:min-invar-implies}.
 	
 	\begin{lemma}
 	\label{lem:measure-zero-preserved}
 	Let $\{ \tau^n, \tau_\lambda\}_{\lambda, n}$ be a $\Lambda$-semibranching function system on $(X, \mu)$.  If $\mu(B) = 0$ then $\mu(\tau^n(B)) = 0$ for any $n \in \N^k$.
 	\end{lemma}
 	\begin{proof}
 Observe that $\tau^n(B) =_{a.e.} \bigsqcup_{d(\lambda) = n} (\tau_\lambda)^{-1}(B)$. Moreover, if $\mu(B) = 0$,
\[ \int_{(\tau_\lambda)^{-1}(B)} \frac{d(\mu \circ \tau_\lambda)}{d\mu} \, d\mu = (\mu \circ \tau_\lambda)(\tau_\lambda^{-1}(B))  = \mu(B) = 0.\]
However, the definition of a $\Lambda$-semibranching function system requires   $\Phi_\lambda = \frac{d(\mu \circ \tau_\lambda)}{d\mu}  > 0$ a.e.~on $D_{s(\lambda)}$.  In other words, $\mu(\tau_\lambda^{-1}(B)) = 0$ for all $\lambda \in \Lambda^n$, so $\mu(\tau^n(B)) = 0$.
 	\end{proof}

Proposition \ref{prop:restriction-gives-SBFS}  below shows that restricting a $\Lambda$-projective system to a subspace $(A, \mu)$ of $(X, \mu)$ will still give a $\lambda$-projective system, as long as the subspace is {invariant}.

\begin{defn} 
\label{def:invariant-set}
Let $(X, \mu)$ be a measure space, and $T: X \to X$ a function. We say that $B\subset X$ is {\em invariant} with respect to $T$ if $\mu(T^{-1}(B)\Delta B)=0$. 
\end{defn}

      Given a measurable subset $A \in \Sigma$ of a measure space $(X, \Sigma, \mu)$, we write $\mu_A := \mu(\, \cdot \cap A)$ for the measure given by restriction to $A$. We take the $\sigma$-algebra of $\mu_A$-measurable sets to be $\{ B \cap A: B \in \Sigma\}$.

\begin{prop}
Suppose there is a $\Lambda$-semibranching function system $\{ \tau_\lambda, \tau^n\}$ on $(X, \mu)$.
If $A\subseteq X$  is invariant with respect to $\tau^n$ for all $n$, and
 $\mu(A \cap D_v)$ is nonzero for all $v$, then the restriction of a $\Lambda$-projective system on $(X, \mu)$  to $(X, \mu_A)$ is again   a $\Lambda$-projective system.
\label{prop:restriction-gives-SBFS}
\end{prop}
\begin{proof}
We first check that if $\{ \tau_\lambda\}_{\lambda \in \Lambda}$ is a $\Lambda$-semibranching function system on $(X, \mu)$, then $\{ \tau_\lambda\}_\lambda$ also gives a $\Lambda$-semibranching function system  on $(X, \mu_A)$.  By hypothesis we have $\mu_A(D_v) > 0$ for all $v \in \Lambda^0$, and for any $n\in \N^k$, 
\[ \mu_A\left( X \backslash \bigcup_{d(\lambda) = n} R_\lambda \right) = \mu \left( A \cap \left(X \backslash \bigcup_{d(\lambda) = n} R_\lambda \right)  \right) \leq \mu \left(X \backslash \bigcup_{d(\lambda) = n} R_\lambda \right) = 0.\]

We now argue that, for any $\lambda \in \Lambda$, the set $Y_\lambda := \{ y \in D_{s(\lambda)} : \frac{d(\mu_A \circ \tau_\lambda)}{d\mu_A}(y) = 0\}$ has $\mu_A$-measure zero. 
{As $A$ is invariant, i.e., $\mu((\tau^{d(\lambda)})^{-1}(A)\Delta A)=0$, the sets 
\[\quad  \{ x \in A: \tau^{d(\lambda)}(x) \not\in A\} =_{a.e.} \{ x \in A: \text{ for all } z \in A, d(\eta) = d(\lambda), x \not= \tau_\eta(z)\} \]
\[  \text{ and } \quad \{ x\not \in A: \tau^{d(\lambda)}(x) \in A\} =_{a.e.} \{ x\not\in A : x = \tau_\eta(z) \text{ for some } z \in A, d(\eta) = d(\lambda)\}\]
have measure zero.

Thus, their intersection with any measurable subset $\tau_\lambda(D)$ of $R_\lambda$ also has measure zero, so the fact that $\tau_\lambda$ is injective a.e.\footnote{because $\tau^{d(\lambda)} \circ \tau_\lambda =_{a.e.} id|_{D_{s(\lambda)}}$} implies that (writing $x = \tau_\lambda(z)$ for $z \in D$)
\[ 0 = \mu(\{ \tau_\lambda(z) \in A : z \in D \backslash A\}) = \mu(\{ \tau_\lambda(z) \not \in A: z \in A \cap D\}).\]
In other words, if $D \subseteq D_{s(\lambda)} \cap A$, then 
\[ 
 \mu(\tau_\lambda(D\cap A)\cap A) = \mu(\tau_\lambda(D\cap A)) = \mu(\tau_\lambda(D)) =  \mu(\tau_\lambda(D) \cap A) .\]
 We conclude that for any measurable set $D \subseteq D_{s(\lambda)} \cap A$,}

\[ 0 < \int_{D} \frac{d(\mu \circ \tau_\lambda)}{d\mu} \,d\mu  = \mu(\tau_\lambda(D)) = \mu(\tau_\lambda(D) \cap A) = \mu_A(\tau_\lambda(D)) = \int_{D} \frac{d(\mu_A \circ \tau_\lambda)}{d\mu_A} \, d\mu_A.\]
As $\mu|_A = \mu_A$, the uniqueness of the Radon-Nikodym derivatives implies that 
\begin{equation}
\label{eq:RN-restriction}
\frac{d(\mu \circ \tau_\lambda)}{d\mu} = \frac{d(\mu_A \circ \tau_\lambda)}{d\mu_A} \quad \text{ a.e.~on } \quad D_{s(\lambda)} \cap A.
  \end{equation} 
  Recall from Condition (2) of Definition \ref{def:SBFS} that
\[X_\lambda := \left\{ x: \frac{d(\mu \circ \tau_\lambda)}{d\mu} (x) = 0 \right\}\]
 has $\mu$-measure 0.  In other words, $Y_\lambda = X_\lambda \cap A$ up to sets of measure zero.  Since $\mu(X_\lambda) = 0$ it follows that $\mu_A(Y_\lambda) = 0$ as claimed.

We have thus established that the maps $\{ \tau_\lambda\}_{\lambda \in \Lambda}$ on $(X, \mu_A)$ satisfy Condition (a) of Definition \ref{def-lambda-SBFS-1}, and Condition (b) holds by construction.  The fact that $\mu_A(Y) \leq \mu(Y)$ for all $Y \subseteq X$ gives us Condition (c), and Condition (d) holds on $(X, \mu_A)$ because we have not changed the definition of any of the maps.  It follows that $\{\tau_\lambda\}_{\lambda \in \Lambda}$ induces a $\Lambda$-semibranching function system on $(X, \mu_A)$.

To see that a $\Lambda$-projective system on $(X, \mu)$ restricts to one on $(X, \mu_A)$, suppose that we have functions $\{ f_\lambda\}_{\lambda \in \Lambda}$ satisfying Definition \ref{def:lambda-proj-system} with respect to $\mu$.  That is, each $f_\lambda$ is supported on $R_\lambda$ and $|f_\lambda|^2  = \frac{d(\mu \circ (\tau_\lambda)^{-1})}{d\mu}$, $\mu$-a.e.~on $R_\lambda$.  

Let $D \subseteq R_\lambda$ be $\mu_A$-measurable; then there exists a $\mu$-measurable set $B \subseteq R_\lambda$ such that $D = A \cap B$.  As $D \subseteq A$, 
\[ \int_D \frac{d(\mu \circ (\tau_\lambda)^{-1})}{d\mu} d\mu = \int_D \frac{d(\mu \circ (\tau_\lambda)^{-1})}{d\mu} d\mu_A.\]
Since $D = A \cap B \subseteq R_\lambda$ and $A$ is invariant,
{$\tau_\lambda^{-1}(D) = \{ x: \tau_\lambda(x) \in D = B \cap A\} $
satisfies 
\begin{align*}
\tau_\lambda^{-1}(D) \backslash A &= \{ x: \tau_\lambda(x) \in D, x \not\in A \} \subseteq \{ x : \tau_\lambda(x) \in A, x \not\in A\}\\
&\subseteq \tau^{d(\lambda)} (\{ y \in R_\lambda \cap A: \tau^{d(\lambda)}(y) \not\in A\}) \subseteq \tau^{d(\lambda)}( A \Delta (\tau^{d(\lambda)})^{-1}(A))
\end{align*}
has measure zero by Lemma \ref{lem:measure-zero-preserved}.
Furthermore, as $D \subseteq B, (\tau_\lambda)^{-1}(D) \backslash (\tau_\lambda)^{-1}(B) = \emptyset$.  We conclude that $\tau_\lambda^{-1}(D) \subseteq_{a.e.} A \cap (\tau_\lambda)^{-1}(B)$. Similarly, 
\begin{align*}
( A \backslash \tau_\lambda^{-1}(A)) \cap \tau_\lambda^{-1}(B) &= \{ x \in A: \tau_\lambda(x) \in B \backslash A\} \\
&= \tau^{d(\lambda)} (\{ y: y \not\in A, \tau^{d(\lambda)}(y) \in A \} \cap B) \subseteq \tau^{d(\lambda)}( (A \Delta (\tau^{d(\lambda)})^{-1}(A) ) \cap B)
\end{align*}
has measure 0, and so 
\[ A \cap \tau_\lambda^{-1}(B ) \backslash \left( \tau_\lambda^{-1}(D) = \tau_\lambda^{-1}(A) \cap \tau_\lambda^{-1}(B)  \right) \]
has measure zero.  Consequently,}
 $(\tau_\lambda)^{-1}(D) =_{a.e.} A \cap (\tau_\lambda)^{-1}(B).$  Consequently, 
\[ \int_D \frac{d(\mu \circ (\tau_\lambda)^{-1})}{d\mu} d\mu = \mu( A \cap (\tau_\lambda)^{-1}(B) ) = \mu_A( (\tau_\lambda)^{-1}(B) ) = \int_D \frac{d(\mu_A \circ (\tau_\lambda)^{-1})}{d\mu_A} d\mu_A.\]
As $D$ was an arbitrary $\mu_A$-measurable set and Radon-Nikodym derivatives are unique, it follows that 
\[ \left.\frac{d(\mu \circ (\tau_\lambda)^{-1})}{d\mu} \right|_A = \frac{d(\mu_A \circ (\tau_\lambda)^{-1})}{d\mu_A}.\]
Consequently, if the functions $\{f_\lambda\}_{\lambda \in \Lambda}$ give a $\Lambda$-projective system on $(X, \mu)$, so that $|f_\lambda|^2 = \frac{(\mu \circ (\tau_\lambda)^{-1})}{d\mu}$, then their restrictions $f_\lambda|_A $ satisfy $\left| f_\lambda|_A\right|^2 =  \frac{d(\mu_A \circ (\tau_\lambda)^{-1})}{d\mu_A}$.  The fact that the restrictions $f_\lambda|_A$ satisfy Condition (b) of Definition \ref{def:lambda-proj-system} is immediate from the assumption that Condition (b) holds for the functions $f_\lambda$.
\end{proof}

\subsection{Measures on the infinite path space: the Carath\'eodory/Kolmogorov extension theorem }

{In this section we will present some  results that will guarantee the existence of (projection-valued)
measures on the infinite path space $\Lambda^\infty$ of a  row-finite source-free $k$-graph. Indeed, it will turn out that
by using the  Carath\'eodory/Kolmogorov extension theorems and their projection-valued analogues, it will be sufficient to define our measures on cylinder sets.} 

Recall that a measure $\mu$ on a measure space $(X, \mathcal{B})$ is $\sigma$-finite if there exists a sequence of subsets $S_n \in  \mathcal{B}$ with $X =\bigcup_n S_n$ and $\mu(S_n)< \infty,\  \forall n.$ 
{Also recall that a  family $\mathcal{S}$  of subsets of a set $X$ is called a semiring of sets if it contains the empty set, $A \cap B \in  \mathcal{S}$ for all $ A, B \in  \mathcal{S}$ and, for every pair of sets $A, B \in\mathcal{S}$ with $A\subseteq B$, the set $B\
	\backslash A$ is the union of finitely many disjoint sets in $\mathcal{S}$. If $X \in \mathcal{S}$, then $ \mathcal{S}$ is called a semialgebra. Semirings and semialgebras canonically generate associated rings and algebras of sets by taking finite unions. In particular  a semiring (resp.~semialgebra) is a ring (resp.~algebra) if and only if is closed under finite unions.}

\begin{thm}[Carath\'eodory/Kolmogorov]  
\label{thm:Cara-extension}\cite[Theorem 1.3.10]{Ash} If 
$\mathcal F_0$ is an algebra\footnote{Ash \cite{Ash} uses the word ``field'' instead of ``algebra.''} of subsets of $X$, and $\mu$ is a countably additive function on $\mathcal F_0$ such that $X$ is $\sigma$-finite with respect to $(\mathcal F_0, \mu)$, then $\mu$ extends uniquely to a measure on the $\sigma$-algebra generated by $\mathcal F_0$.
\end{thm}

{We also note the following projection-valued measure extension of the above result:

\begin{lemma}
	\label{lem:kolmogorov-projection-valued-measures} \cite[Theorem 7]{berberian} If $\mathcal{F}$ is a bounded projection-valued measure defined on a ring of sets  $\mathcal{R}$, there exists one and only one (necessarily bounded) 
	projection-valued measure $\mathcal{E}$ on $\sigma(\mathcal{R})$, the $\sigma$-ring generated by $\mathcal{R},$ such that $\mathcal{E}$ is an extension of $\mathcal{F}$. If $\mathcal{R}$ is an algebra, then  $ \sigma(\mathcal{R})$ equals  the $\sigma$-algebra  generated by $\mathcal{R}$.
\end{lemma}

We will apply the above results to construct (projection-valued) measures  on $\Lambda^\infty$.
In particular, in our applications   we will often take $X = v\Lambda^\infty$
for a fixed vertex $v \in \Lambda^0$, and the algebra $\mathcal{S}$ to be the collection of sets formed by taking finite intersections  and unions
 of the cylinder sets $Z(\lambda)$ with $r(\lambda) = v$.
To check that $\mathcal S$ is an algebra, 
 notice that $X = v\Lambda^\infty = Z(v)$ is in $\mathcal S$ {and use the following lemma, whose proof is a straightforward application of the definitions given above.
 
\begin{lemma}\label{carnival} Let $\Lambda$ be a row-finite k-graph and $v$ be a vertex in $\Lambda^0$. If $\alpha, \beta \in v\Lambda$ then:
\begin{enumerate}
	\item  $Z(\alpha) \cap Z(\beta) = \bigsqcup_{(\lambda, \nu) \in \Lambda^{\min}(\alpha,\beta)} Z(\alpha \lambda)$ is a finite disjoint union of cylinder sets;
	\item  $Z(\alpha) \backslash Z(\beta) = Z(\alpha) \backslash \bigsqcup_{\xi \in MCE(\alpha, \beta)} Z(\xi) = \bigsqcup \{ Z(\alpha \lambda): d(\alpha \lambda) = d(\alpha) \vee d(\beta) \text{ but } \alpha \lambda \not\in MCE(\alpha, \beta) \}$  
	 is also a finite disjoint union of cylinder sets;
	\item $Z(\alpha) \cup Z(\beta) $ is therefore also a finite disjoint union of cylinder sets.
\end{enumerate}
\end{lemma}
}
}

Before discussing projection-valued measures on $\Lambda^\infty$, we 
{pause to reassure the reader that there do indeed exist real-valued  measures on the infinite path space 
of  row-finite $k$-graphs. One approach to constructing such measures is to find  a vector  $\xi \in \R_{> 0}^{\Lambda^0}$ which is an eigenvector for each adjacency matrix $A_i$ of $\Lambda$.  Given such an eigenvector, write $\beta_i$ for the eigenvalue of $A_i$ with respect to $\xi$, and for $n = (n_1, \ldots, n_k) \in \Z^k$, write $\beta^n := \beta_1^{n_1} \cdots \beta_k^{n_k}$.  Then, if we  define 
\[ \mu(Z(\lambda)) := \beta^{-d(\lambda)} \xi_{s(\lambda)},\]
one can compute that $\mu$ is countably additive on the algebra $\mathcal S$ of finite unions of cylinder sets, and hence, by Theorem \ref{thm:Cara-extension}, induces a measure on the $\sigma$-algebra generated by the cylinder sets. This is the content of the next theorem, which arose from discussions with Sooran Kang.} {
\begin{thm}
	Suppose that $\Lambda$ is a row-finite $k$-graph with no sources.
	If there exists a vector $\xi \in \R^{\Lambda^0}_{>0}$ which is an eigenvector for each adjacency matrix $A_i$ of $\Lambda$, then the formula 
	\begin{equation}
	 \mu(Z(\lambda)) := \beta^{-d(\lambda)} \xi_{s(\lambda)}
	 \label{eq:mu}
	\end{equation}
	defines a measure on the Borel $\sigma$-algebra of $\Lambda^\infty$.
	\label{thm:meas-from-eigenvect}
\end{thm}
\begin{proof}
We will first show that $\mu$ is well defined and finitely additive on cylinder sets; that is, if $Z(\lambda) = \sqcup_{i=1}^p Z(\eta_i)$ then $\sum_{i=1}^p \beta^{-d(\eta_i)} \xi_{s(\eta_i)} = \beta^{-d(\lambda)} \xi_{s(\lambda)}$.

	Suppose  $Z(\lambda) = \sqcup_{i=1}^p Z(\eta_i)$.  
	Since $Z(\lambda)=\sqcup_{i=1}^pZ(\eta_i)$, $\MCE(\lambda,\eta_i)\ne \emptyset$ for all $1\le i\le p$. (In fact, for any $n \geq d(\lambda)\vee d(\eta_i)$ and for any $\beta_i^j \in s(\eta_i)\Lambda^{n-d(\eta_i)}$, 
	the fact that $Z(\lambda) = \sqcup_{i=1}^p Z(\eta_i)$ implies that  $\eta_i \beta_i^j$ is an extension of $\lambda$.  Consequently, there must exist a corresponding $\alpha_i^j \in s(\lambda) \Lambda^{n-d(\lambda)}$ so that $\lambda \alpha_i^j = \eta_i \beta_i^j$.  In other words, 
	for any $n\geq d(\lambda) \vee d(\eta_i)$, there is a bijection between $\Lambda^{\min}(\lambda, \eta_i)$ and $s(\eta_i)\Lambda^{n - d(\eta_i)}$.)
	
	Since $\Lambda$ is row-finite, for each $i$, the set $\Lambda^{\min}(\lambda, \eta_i)$ is finite; write its elements as $\{(\alpha_i^j, \beta_i^j)\}_{j\in J}$, where $J$ is a finite index set. Let
	\[
	n=\vee_{i,j}d(\lambda\alpha_i^j)=\vee_{i,j}d(\eta_i\beta^j_i).
	\]
	Then we have
	\begin{equation*}
		Z(\lambda)=\bigsqcup_{i=1}^p\bigsqcup_{(\alpha_i^j, \beta_i^j)\in \Lambda^{\min}(\lambda,\eta_i)} \bigsqcup_{\xi_{ij}\in s(\alpha_i^j)\Lambda^{n-d(\lambda\alpha_i^j)}}Z(\lambda\alpha_i^j\xi_{ij}).
	\end{equation*}
	
	Observe that, for each $i,j$, $d(\alpha^i_j \xi_{ij}) = n - d(\lambda \alpha^i_j) + d(\alpha^i_j) = n-d(\lambda)$.  Moreover, 
	for any $ m\in \N^k,$  $Z(\lambda)=\sqcup_{\gamma\in s(\lambda)\Lambda^m}Z(\lambda\gamma)$.  Taking $m=n-d(\lambda)$ tells us that 
\begin{equation}
	\bigsqcup_{\gamma\in s(\lambda)\Lambda^m}Z(\lambda\gamma)=Z(\lambda) = \bigsqcup_{i=1}^p\bigsqcup_{(\alpha_i^j, \beta_i^j)\in \Lambda^{\min}(\lambda,\eta_i)} \bigsqcup_{\xi_{ij}\in s(\alpha)i^j)\Lambda^{n-d(\lambda\alpha_i^j)}}Z(\lambda\alpha_i^j\xi_{ij}).
	\label{eq:finite_2}
	\end{equation}
	Since both sides of the above equality are disjoint unions of cylinder sets of the same degree, the list of cylinder sets on the left must be precisely equal to the list of cylinder sets on the right.  That is, each cylinder set $Z(\lambda \gamma)$ must equal $Z(\lambda \alpha_i^j \xi_{ij})$ for precisely one path $\alpha_i^j \xi_{ij}$.

From the fact that $\xi$ is an eigenvector for each adjacency matrix $A_i$ with eigenvalue $\beta_i$, we easily compute that for any $m \in \N^k$,
	\begin{equation}\label{eq:fiinite_1}
		\mu(Z(\lambda))=\sum_{\gamma\in s(\lambda)\Lambda^m}\mu(Z(\lambda\gamma)).
	\end{equation}	
It now follows from Equations \eqref{eq:fiinite_1} and \eqref{eq:finite_2} that
	\begin{equation}\label{eq:finite_A1}
		\mu(Z(\lambda))=\sum_{i=1}^p\sum_{(\alpha_i^j, \beta_i^j)\in \Lambda^{\min}(\lambda,\eta_i)}\sum_{\xi_{ij}\in s(\alpha_i^j)\Lambda^{n-d(\lambda\alpha_i^j)}} \mu(Z(\lambda\alpha_i^j\xi_{ij})).
	\end{equation}
	Since $\lambda\alpha_i^j=\eta_i\beta_i^j$, we get
	\begin{equation}\label{eq:finite_A2}
		\mu(Z(\lambda))=\sum_{i=1}^p\sum_{(\alpha_i^j, \beta_i^j)\in \Lambda^{\min}(\lambda,\eta_i)}\sum_{\xi_{ij}\in s(\alpha)i^j)\Lambda^{n-d(\eta_i\beta_i^j)}} \mu(Z(\eta_i\beta_i^j\xi_{ij})).
	\end{equation}
	On the other hand, for fixed $1\le i\le p$,
	\begin{equation}\label{eq:finite_4}
		Z(\eta_i)=\bigsqcup_{(\alpha_i^j, \beta_i^j)\in \Lambda^{\min}(\lambda,\eta_i)}\bigsqcup_{\xi_{ij}\in s(\alpha_i^j)\Lambda^{n-d(\eta_i\beta_i^j)}}Z(\eta_i\beta_i^j\xi_{ij}),
	\end{equation}
	where $n=\vee_{i,j}d(\lambda\alpha_i^j)=\vee_{i,j}d(\eta_i\beta^j_i)$. Again, $d(\beta^j_i \xi_{ij}) = n - d(\eta_i \beta^j_i) + d(\beta^j_i) = n - d(\eta_i)$ is the same for all $\xi_{ij}$.
	In other words, we can apply Equation \eqref{eq:fiinite_1} to $\eta_i$ instead of $\lambda$, using the decomposition of $Z(\eta_i)$  from Equation \eqref{eq:finite_4} and setting $m= n -d(\eta_i)$.  It follows that  
	\begin{equation}
	\label{finite_B}
	\mu(Z(\eta_i)) = \sum_{(\alpha_i^j, \beta_i^j)\in \Lambda^{\min}(\lambda,\eta_i)}\sum_{\xi_{ij}\in s(\alpha_i^j)\Lambda^{n-d(\eta_i\beta_i^j)}} \mu(Z(\eta_i\beta_i^j\xi_{ij})).
	\end{equation}
	Now combining \eqref{eq:finite_A1}, \eqref{eq:finite_A2} and \eqref{finite_B}, we obtain
	\[\begin{split}
		\sum_{i=1}^p \mu(Z(\eta_i))&=\sum_{i=1}^p \sum_{(\alpha_i^j, \beta_i^j)\in \Lambda^{\min}(\lambda,\eta_i)}\sum_{\xi_{ij}\in s(\alpha_i^j)\Lambda^{n-d(\eta_i\beta_i^j)}} \mu(Z(\eta_i\beta_i^j\xi_{ij}))\\
		&=\sum_{i=1}^p\sum_{(\alpha_i^j, \beta_i^j)\in \Lambda^{\min}(\lambda,\eta_i)}\sum_{\xi_{ij}\in s(\alpha_i^j)\Lambda^{n-d(\lambda\alpha_i^j)}} \mu(Z(\lambda\alpha_i^j\xi_{ij}))\\
		&=\mu(Z(\lambda)).
	\end{split}\]
	 In other words, $\mu$ is indeed well-defined and finitely additive on  cylinder sets.
	 
	 We now show that $\mu$ is countably additive, and hence a measure on the algebra $\mathcal S$ of finite disjoint unions of cylinder sets.
By construction, if $S \in \mathcal S$ and $S = \sqcup_{i \in \N} Z(\lambda_i)$, we are defining $\mu(S) = \sum_i \mu(Z(\lambda_i))$.  To see that $\mu$ is a measure, we  merely need to check that $\mu$ is additive on countable disjoint unions.  Thus, suppose $\bigsqcup_{i \in \N} Z(\lambda_i)= \bigsqcup_{j\in \N} Z(\eta_j)$.  For each fixed $i \in \N$ we have 
	\[Z(\lambda_i) = \bigsqcup_{j\in \N} Z(\lambda_i) \cap Z(\eta_j),  \qquad \text{ and } \qquad Z(\lambda_i) \cap Z(\eta_j) = \bigsqcup_{\zeta \in \text{MCE}(\lambda_i, \eta_j)} Z(\zeta).\]
	The fact that $\Lambda$ is row-finite ensures that $\text{MCE}(\lambda_i, \eta_j)$ is finite for all $i, j$, and also that each cylinder set $Z(\lambda)$ is compact and open.  It follows that, since $\bigsqcup_{j\in \N} Z(\eta_j)$ is a cover for $Z(\lambda_i)$, there are only finitely many indices $j$ such that $Z(\lambda_i) \cap Z(\eta_j) \not= \emptyset.$

	Therefore, using the finite additivity of $\mu$, we have
	\begin{align*}
		\sum_{i\in \N} \mu(Z(\lambda_i)) &= \sum_{i\in \N} \sum_{j \in \N} \mu(Z(\lambda_i) \cap Z(\eta_j)) \\
		&= \sum_{i, j \in \N} \sum_{\zeta \in \text{MCE}(\lambda_i, \eta_j)} \mu(Z(\zeta)) \\
		&= \sum_{j\in \N} \mu(Z(\eta_j)),
	\end{align*}
	as desired.  Note that we are able to interchange the order of the summation over $i$ and $j$ since, for fixed $i$ (or equivalently for fixed $j$), only finitely many of the intersections $Z(\lambda_i) \cap Z(\eta_j)$ are non-empty.

Now, we use Carath\'eodory's Theorem (Theorem \ref{thm:Cara-extension}) to extend $\mu$ uniquely to give a measure (also denoted $\mu$) on the Borel $\sigma$-algebra of $\Lambda^\infty$, as desired. 
\end{proof}
}

{The above analysis begs the question of when  the adjacency matrices of $\Lambda$  admit a common positive eigenvector $\xi$.  When $\Lambda$ is finite and strongly connected, \cite[Corollary 4.2]{aHLRS3} guarantees that a unique such eigenvector (of $\ell^1$ norm 1) exists.  For row-finite (not necessarily finite) $k$-graphs, if $k=1$, 
Thomsen identified in \cite{Thom} when an infinite directed graph $\Lambda$ will admit such an eigenvector $\xi$. While we anticipate that much of Thomsen's analysis could be extended to the setting of higher-rank graphs, for the moment we simply present one example where this can be done.}

\begin{example}
\label{prop:cofinal-recurrent-example}
Define matrices 
\[ A_1 := \begin{pmatrix}
 1 & 0 & 0 &1 \\ 1 & 0 & 0 &1 \\ 0 & 1 & 1 &0 \\ 0 & 1 & 1 & 0 
\end{pmatrix}, \quad
A_2 := \begin{pmatrix}
1 & 0 & 1 & 0 \\ 1 & 0 & 1 & 0 \\ 0 & 1 & 0 &1 \\ 0 & 1 & 0 &1
\end{pmatrix} \]
Let $I$ denote the $4 \times 4 $ identity matrix, and define 
\[S := \begin{pmatrix}
A_1 &  I & 0 & 0 & 0 & \ldots \\
I & A_1 & I & 0 & 0 & \ldots \\
0 & I & A_1 & I & 0 & \ldots \\
%0 & 0 & I & A_1 & I & 0 & \ldots \\
\vdots & &  & \ddots  & & \ldots 
\end{pmatrix},
\quad
T := \begin{pmatrix}
A_2 &  I & 0 & 0 & 0 & \ldots \\
I & A_2 & I & 0 & 0 & \ldots \\
0 & I & A_2 & I & 0 & \ldots \\
%0 & 0 & I & A_1 & I & 0 & \ldots \\
\vdots &  & & \ddots  & & \ldots 
\end{pmatrix}
\]
Let $\Lambda$ be the infinite 2-graph with vertex matrices $S, T$.  Notice that $\Lambda$ consists of countably many copies  $\{\Gamma_n\}_{n\in \N}$ of the Ledrappier 2-graph $\Gamma$ (cf.~Example 5.4 of \cite{FGKPexcursions});  consecutive copies are linked by four ``forward'' edges and four ``backward'' edges of each color. 
Moreover, the vector  $\xi = (1_4, 2_4, 3_4, 4_4, 5_4, \ldots)$ given by $\xi_{4\ell + m} = \ell+1$ (if $0 \leq m \leq 3$) is an eigenvector for both $S$ and $T$, with eigenvalue $4$.

Thus, using the Carath\'eodory/Kolmogorov Extension Theorem (Theorem
\ref{thm:Cara-extension} above) we obtain a  measure  $\mu$ on $\Lambda^\infty$
which extends the measure defined on cylinder sets by
\[
\mu( Z( \lambda )  ) =  (4,4)^ { - d(\lambda)} \ \xi_{s(\lambda)},\quad \forall \lambda\in \Lambda.
\]

\end{example}

\section{{Monicity}}
\label{sec:monicity}

{The main results of this section are Theorem \ref{thm-characterization-monic-repres} and Theorem \ref{thm:range-sets-generate-sigma-alg-in-monic}, which describe when a $\Lambda$-projective representation is equivalent to one arising from the infinite path space of $\Lambda$. Theorem~\ref{thm-characterization-monic-repres} gives a  representation-theoretic description, while Theorem~\ref{thm:range-sets-generate-sigma-alg-in-monic} explains the equivalence at the level of the measure space $(X, \mu)$ underlying the $\Lambda$-projective representation.   In both cases, versions of these results were known for finite $k$-graphs, but Theorem~\ref{thm:range-sets-generate-sigma-alg-in-monic} is stronger than the previously established results, even in the finite case.  We conclude this section by providing a necessary condition for a $\Lambda$-projective representation to be monic in Proposition~\ref{prop:vertex-prop-implies-no-monic-repns}, and we discuss the relationship of this condition to the existence of cycles without entry and their higher-rank generalizations.}

We first introduce, given a representation of $C^*(\Lambda),$ 
an associated  projection-valued measure on $\Lambda^\infty,$ which will 
prove to be an invaluable tool in this section.

{Enumerate the vertices of the k-graph, say $\{v_i\}_{i\in \N}$. Recall that the $\sigma$-algebra $\mathcal{A}$ of the disjoint union
$\Lambda^\infty= \bigsqcup_{n \in \N} v_n \Lambda^\infty$ is defined to be
\[
\mathcal{A} = \Big\{  A\subseteq \Lambda^\infty\hbox{ such that } A\cap v_j\Lambda^\infty \hbox{ Borel, }\forall j\in \N\Big\}.
\]}

  \subsection{Projection-valued measures}
  \label{sec:Proj-valued-measures}
  
  Lemma~\ref{lem:kolmogorov-projection-valued-measures} is central to the proof of the following Proposition.
  
 \begin{prop} 
   \label{prop-proj-valued-measure-gen-case}
   Let $\Lambda$ be a row-finite $k$-graph with no sources.
   Given a representation $\pi: C^*(\Lambda ) \to \mathcal{H}$,   $\{\pi(s_\lambda)=t_\lambda\}_{\lambda\in\Lambda},$ of a $k$-graph $C^*$-algebra $C^*(\Lambda)$ on a separable Hilbert space $\mathcal{H}$, there is a unique regular projection-valued measure $P$ on the Borel $\sigma$-algebra  $\mathcal{B}_o(\Lambda^\infty)$ of the infinite path space $\Lambda^\infty$ which is defined on cylinder sets  by
  \begin{align}
   P(Z(\lambda)) &= t_\lambda t_\lambda^*  \quad\;\;\text{for all}\;\; \lambda \in \Lambda.
 \end{align}
 Moreover, the restriction $\pi$ of the representation $\{t_\lambda \}_{\lambda \in \Lambda}$ to the subalgebra $C_0(\Lambda^\infty)$ is given by 
   \begin{equation}
   \label{eq:integral-decomposition}
   \pi(f) = \int_{\Lambda^\infty} f(x) dP(x),\ \forall f \in C_0(\Lambda^\infty).
   \end{equation}
   \end{prop}
     
   \begin{proof} 
   First we will deal with the compact spectrum case by restricting to $v\Lambda^\infty$, where $v \in \Lambda^0$. 
   
   Denote by $\mathcal S$ the algebra of finite unions of cylinder sets  $\{ Z(\lambda): \lambda \in v\Lambda\}$. Given $A\in \mathcal S$,  write it as a finite union of disjoint cylinder sets (this can be done by Lemma~\ref{carnival}), say $A =  \bigsqcup_{i=1}^{n}  Z(\lambda_i)$, and define $P(A) = \sum_{i=1}^{n} t_{\lambda_i} t_{\lambda_i}^*$.  Without loss of generality, since $\Lambda$ is row-finite and source-free, we may assume that each $\lambda_i$ has the same degree.  Thus, (CK4) guarantees that the projections making up $P(A)$ are mutually orthogonal, so that $P(A)$ is a projection for any such $A$.

   Next we prove that $P$ is well defined (and hence a projection-valued measure on $\mathcal S$). 
Suppose that $A$ can be written in two ways as a finite disjoint union of cylinder sets, $A = \bigsqcup Z(\lambda_i)= \bigsqcup Z(\eta_j) $. Then for each fixed $i \in \N$ we have 
 \[Z(\lambda_i) = \bigsqcup_j Z(\lambda_i) \cap Z(\eta_j),\]
 and that 
 \[Z(\lambda_i) \cap Z(\eta_j) = \bigsqcup_{\zeta \in \text{MCE}(\lambda_i, \eta_j)} Z(\zeta).\]
 The fact that $\Lambda$ is row-finite ensures that $\text{MCE}(\lambda_i, \eta_j)$ is finite for all $i, j$.

Set $m_i = \bigvee_j d(\lambda_i) \vee d(\eta_j) \in \N^k$ to be the coordinate-wise maximum degree of all elements in $\{\text{MCE}(\lambda_i, \eta_j)\}_j$, and write $n_{i,j} = m_i - (d(\lambda_i) \vee d(\eta_j)) \in \N^k$.  Note that since $\bigsqcup_i Z(\lambda_i) = \bigsqcup_j Z(\eta_j)$, we in particular have that $Z(\lambda_i) \subseteq \bigsqcup_j Z(\eta_j)$.  Thus every $\alpha \in s(\lambda_i) \Lambda^{m_i - d(\lambda_i)}$  must satisfy $Z(\lambda_i \alpha) \subseteq \bigsqcup_j Z(\eta_j)$.  By the definition of $m_i$, we then have that  $\lambda_i \alpha= \zeta \beta$ for a unique $\zeta \in \text{MCE}(\lambda_i, \eta_j), \ \beta \in s(\zeta)\Lambda^{n_{i,j}}.$ Moreover, if $\zeta \in \text{MCE}(\lambda_i, \eta_j), \ \beta \in s(\zeta)\Lambda^{n_{i,j}}$ then the fact that $\zeta $ extends $\lambda_i$ and our choice of the degrees of $\zeta, \beta$ imply that $ \zeta \beta = \lambda_i \alpha$ for some $\alpha \in s(\lambda_i) \Lambda^{m_i - d(\lambda_i)}$.

Since $t_{\lambda}$ is a partial isometry for all $\lambda \in \Lambda$, it follows {by (CK3)}  that 
\[t_\lambda = t_\lambda t_\lambda^* t_\lambda = t_\lambda \pi (s_{s(\lambda)}) =t_\lambda \sum_{\beta \in s(\lambda)\Lambda^n} \pi(s_\beta s_\beta^*)  = \sum_{\beta \in s(\lambda)\Lambda^n} t_{\lambda \beta} t_\beta^*\]
for any $n \in \N^k$.  It now follows {by using (CK3), (CK4), and (CK2)} that
 \begin{align*}
 P(\bigsqcup_{i} Z(\lambda_i)) &= \sum_i t_{\lambda_i} t_{\lambda_i}^* = \sum_i t_{\lambda_i } \left(\sum_{\alpha \in s(\lambda_i)\Lambda^{m_i-d(\lambda_i)}} t_\alpha t_{\alpha}^* \right) t_{\lambda_i}^*  = \sum_i \sum_{\alpha \in s(\lambda_i)\Lambda^{m_i-d(\lambda_i)} } t_{\lambda_i \alpha }  t_{\lambda_i \alpha}^* \\
 &= \sum_i \sum_j \sum_{\zeta \in \text{MCE}(\lambda_i, \eta_j)} \sum_{\beta \in s(\zeta) \Lambda^{n_{i,j}}} t_{\zeta\beta} t_{\zeta\beta}^*  = \sum_i \sum_j \sum_{\zeta \in \text{MCE}(\lambda_i, \eta_j)} t_\zeta t_\zeta^*
.
 \end{align*}
 Since all of the sums in question are finite, we can rearrange the order of summation; then, by a symmetric argument to the one given above (replacing $m_i$ with $n_j = \bigvee_i d(\eta_j) \vee d(\lambda_i)$ and $\alpha \in s(\lambda_i) \Lambda^{m_i- d(\lambda_i)}$ with $\gamma \in s(\eta_j) \Lambda^{n_j - d(\eta_j)}$) 
 we see that 
 \[ P(\bigsqcup_i Z(\lambda_i) ) = \sum_j \sum_i \sum_{\zeta \in \text{MCE}(\lambda_i, \eta_j)} t_\zeta t_\zeta^* = \sum_j t_{\eta_j} t_{\eta_j}^* = P(\bigsqcup_j Z(\eta_j)).\]

 By Lemma~\ref{lem:kolmogorov-projection-valued-measures} $P$ extends to a unique measure on the Borel $\sigma$-algebra of $v\Lambda^{\infty}$.
 Now enumerate the vertices of $\Lambda$, say $\{v_i\}_{i\in \N}$. Recall that the $\sigma$-algebra $\mathcal{A}$ of the disjoint union
   $\Lambda^\infty= \bigsqcup_{n \in \N} v_n \Lambda^\infty$ is defined to be
   \[
  \mathcal{A} = \Big\{  A\subseteq \Lambda^\infty\hbox{ such that } A\cap v_j\Lambda^\infty \hbox{ Borel, }\forall j\in \N\Big\}.
   \]
   Define the wanted  projection-valued measure  $P$ on $ \Lambda^\infty$ by
   
   \begin{equation}
   \label{eq:def-proj-val-meas-disj-union}
   P(A )= \sum_j P(A\cap v_j\Lambda^\infty ) .
   \end{equation}
   Note that  $P(A )= \sum_j P(A\cap v_j\Lambda^\infty )$  converges in the weak or strong operator topologies.
   
   Moreover, if  we take a compact set $C \subseteq \Lambda^\infty$ Borel, then  $C  $ is included in a finite union
   of the sets $v_j\Lambda^\infty,$ and so the projection-valued measure defined in Equation 
    \eqref{eq:def-proj-val-meas-disj-union}, when evaluated at $C$, coincides with the projection-valued measure  coming from taking disjoint unions of  finitely many   $v_j\Lambda^\infty $ in this case.
       The uniqueness of the measure follows from the uniqueness of the extension to  each $ \mathcal{B}_o(v_n \Lambda^\infty),$ and from the fact that every operator on $L^2(\Lambda^\infty)$ must be of the form given in Equation  \eqref{eq:def-proj-val-meas-disj-union}, since the spaces $L^2(v_n \Lambda^\infty)$ are orthogonal.   {The regularity of the measure $P$ follows from the fact that $\Lambda^\infty$ is $\sigma$-compact and metrizable, so all Borel sets are Baire sets.  Consequently, \cite[Theorem 18]{berberian} implies that $P$ is regular.}

    Next we show the integral description of the representation. Following \cite[Section 37]{Halmos}, define 
    {for $f$ bounded}
\[ \int_{\Lambda^\infty} f(x) dP(x) := A_f,\]
   where $A_f \in B(\H)$ is the unique operator such that for all $\xi, \zeta \in \H$, $\langle A_f \xi, \zeta \rangle = \int_{\Lambda^\infty} f(x) dP_{\xi, \zeta}(x)$, for the complex measure $P_{\xi,\zeta}$ defined by 
   \begin{equation}
   \label{eq:def-proj-valued-measure-via-measures}
   P_{\xi,\zeta}(Z(\lambda)) = \langle t_\lambda t_\lambda^* \xi, \zeta \rangle.
   \end{equation}
 {Note that  our definition  of $A_f$ ensures that for all $\gamma \in \Lambda,$
   \[ A_{\chi_{Z(\gamma)}} = \pi(\chi_{Z(\gamma)} ).\]}
 Suppose $f_n \to f \in C_0(\Lambda^\infty)$ and $f_n = \sum_{i=1}^{k_n} \alpha_i^n \chi_{Z(\lambda_i)}$, where $Z(\lambda_i) \cap Z(\lambda_j) = \emptyset$ for all $i \not= j$. Then for all $\epsilon > 0$, there is $N \in \N$ such that if $n \geq N$ we have $\sup \{ |f(x)|: x \not\in \bigcup_{i=1}^{k_n} Z(\lambda_i)\} < \epsilon\}$ and $
 \sup \{ |f(x) - \alpha_i^n|: x \in Z(\lambda_i) \} < \epsilon$, and so 
 \[ |\langle( A_f - A_{f_n})\xi, \zeta \rangle | = \left| \int_{\Lambda^\infty} (f- f_n )dP_{\xi, \zeta}\right| \leq \epsilon | \langle \xi, \zeta\rangle |.\]

{Since our  construction of $f_n$ ensures that 
$ A_{f_n} = \pi(f_n)$, the above inequality becomes }
\[ | \langle (A_f - \pi(f_n)) \xi, \zeta \rangle | \leq \epsilon | \langle \xi, \zeta \rangle |\]
for all $\xi, \zeta \in \H$ and all $n \geq N$.  It follows that
the sequence $(\pi(f_n))_{n\in \N}$ converges to $A_f$ in norm: if $n\geq N$,

\begin{align*}
\| A_f - \pi(f_n)\|^2& =\| A_f - A_{f_n}\|^2  = \sup_{\| \xi\| = 1} \| (A_f - A_{f_n})\xi\|^2 = \sup_{\| \xi\| =1} \langle (A_f - A_{f_n}) \xi, (A_f - A_{f_n})\xi\rangle  \\
& \leq \sup_{\| \xi \| = 1} \sup \left\{ | \langle (A_f - A_{f_n})\xi, \zeta\rangle | : \| \zeta\| = \|  (A_f - A_{f_n})\xi \| \right\} \\
&\leq \sup_{\| \xi\| =1} \sup \{ \epsilon | \langle \xi, \zeta \rangle | : \|\zeta\| = \|  (A_f - A_{f_n})\xi \| \} \leq \sup_{\| \xi\| = 1} \epsilon \| \xi\| \, \| (A_f - A_{f_n})\xi \| \\
& = \epsilon \| A_f - A_{f_n}\| ,
\end{align*}
and consequently $\| A_f - A_{f_n}\|  \leq \epsilon$ for large enough $n$.

On the other hand, the continuity of $\pi$ implies that $\pi(f_n) \to \pi(f)$.  It follows that 
\[ A_f = \int_{\Lambda^\infty} f(x) dP(x) = \pi(f),\]
as claimed.
\end{proof}

\begin{rmk}
In \cite[Proposition 3.8]{FGJKP-monic} the existence of the projection valued measure $P$ was proved for finite $k$-graphs using the Kolmogorov extension theorem. 
Proposition \ref{prop-proj-valued-measure-gen-case} above gives an alternative approach to the proof of the existence of this projection valued measure. 
\end{rmk}

  We now record some properties of the projection-valued measure associated to a representation of $C^*(\Lambda)$.  The proofs are very similar to the proofs recorded in \cite{FGJKP-monic} in the case of finite $k$-graphs.

   \begin{prop} (\cite[Proposition 3.9 and Definition 4.1]{FGJKP-monic})
   \label{prop-atomic-basic-equns} Let $\Lambda$ be a row-finite, source-free $k$-graph, and fix a representation $\{t_\lambda:\lambda \in \Lambda\}$ of $C^*(\Lambda)$.
   \begin{itemize}\label{prop:pvm-properties}
   \item[(a)] For $\lambda, \eta \in\Lambda$ with $s(\lambda)=r(\eta)$, we have $t_\lambda P(Z(\eta)) t_\lambda^*=P(\sigma_\lambda(Z(\eta))) = P(Z(\lambda \eta))$; 
   
   \item[(b)] For any fixed $n \in \N^k$, we have
    \[
    \sum_{\lambda \in  s(\eta) \Lambda^n} t_\lambda P(\sigma_\lambda^{-1}(Z(\eta))) t_\lambda^* = P(Z(\eta)) ;
    \]
    \item[(c)]  For any $\lambda,\eta \in \Lambda$ with $r(\lambda)=r(\eta)$, we have $t_\lambda P(\sigma_\lambda^{-1}(Z(\eta))) = P(Z(\eta)) t_\lambda$;
    \item[(d)] When $\lambda\in\Lambda^n$, we have $t_\lambda P( Z(\eta)) =  P((\sigma^n)^{-1}(Z(\eta))) t_\lambda$.
   \end{itemize}
   \end{prop}

  \subsection{Monic representations}
\label{sec:monic-representations}  
  We now observe that the definition of 
  monic representation from \cite{FGJKP-monic}, originally given in the context of finite $k$-graphs,  also makes sense for row-finite graphs. 
  
  \begin{defn}\label{def:lambda-monic-repres}(cf.~\cite[Definition 4.1]{FGJKP-monic})
   Let $\Lambda$ be a  row-finite $k$-graph with no sources.
    A representation $\{ t_\lambda\,:\, \lambda\in\Lambda\}$ of $\Lambda$ on a Hilbert space $\H$ is called \emph{monic} if $t_\lambda \not= 0$ for all $\lambda \in \Lambda$, and there exists a vector $\xi\in \mathcal{H}$ such that
   \[
   \overline{\text{span}}\{ t_\lambda t_\lambda^* \xi : \lambda \in \Lambda\} = \mathcal{H}.
   \]
  We say that such a vector $\xi$ is a {\em monic} vector for the representation.
   \end{defn}
   Notice that a monic vector for $\{ t_\lambda \}_\lambda$ is a  cyclic vector for the restriction $\pi$ of the representation generated by $\{t_\lambda\}_{\lambda \in \Lambda}$ to $C_0(\Lambda^\infty)$.
We  therefore obtain  a Borel measure $\mu_\pi$ on $\Lambda^\infty$ given by
\begin{equation}\label{eq:monic_measure}
\mu_\pi(Z(\lambda))=\langle \xi, P(Z(\lambda))\xi\rangle=\langle \xi, t_\lambda t^*_\lambda\xi \rangle.
\end{equation}

An important feature of a monic vector is that its support must have full measure (cf. \cite[Example 4.7]{FGJKP-monic}). 

\begin{prop}\label{supportcyclic} Let $(X, \mu)$ be a $\sigma$-finite measure space.  If a $\Lambda$-projective representation of $C^*(\Lambda)$ on $L^2(X, \mu)$ is monic then the support of  any of its monic vectors $\xi$ can differ from $X$ by at most a set of measure 0.
\end{prop}
\begin{proof}
 Recall that if $\{t_\lambda: \lambda \in \Lambda\}$ is a $\Lambda$-projective representation, then $t_\lambda t_\lambda^* = M_{\chi_{R_\lambda}}$.  Moreover, 
  the support of any function in $\overline{\text{span}} \{ \chi_{R_\lambda} \xi: \lambda \in \Lambda \} $ is contained in the support of $\xi$. So, if $X = \text{supp}(\xi) \sqcup S$ for a set $S$ with positive measure,   the function $\chi_A$, where $A \subseteq S$ is any measurable set of finite positive measure in the complement of the support of $\xi$, 
is not in $\overline{\text{span}} \{ \chi_{R_\lambda} \xi: \lambda \in \Lambda \} $. Consequently, $\{t_\lambda\}_\lambda$ cannot be a monic $\Lambda$-projective representation if $S = X \backslash \text{supp}(\xi)$ has positive measure.
\end{proof}

Note that  the hypothesis that $(X, \mu)$ be $\sigma$-finite is necessary to guarantee the existence of  a set $A$ as in the above proof. If we assume $\mu(R_v) < \infty$ for all $v$ then $\sigma$-finiteness is automatic.

The following theorem  is the main result of this section, which generalizes the finite case given in \cite[Theorem 4.2]{FGJKP-monic}. The condition  $\mu(Z(v_n)) < \infty$ for all $n$ in the second part of the theorem below cannot be removed, see Remark \ref{rmk:finite-vertex-condition}.

  \begin{thm}
  \label{thm-characterization-monic-repres} 
  Let $\Lambda$ be a row-finite  $k$-graph with no sources.  
  If $\pi: C^*(\Lambda) \to \mathcal{H}$, with $\pi(s_\lambda)=t_\lambda$,  is a monic representation of $C^*(\Lambda)$ on a Hilbert space $\H$,
  then $\{t_\lambda\}_{\lambda\in \Lambda}$ is unitarily equivalent to a representation $\{S_\lambda\}_{\lambda\in \Lambda}$ associated to a $\Lambda$-projective system on $(\Lambda^\infty, \mu_\pi)$ associated to the standard coding and prefixing maps $\sigma^n, \sigma_\lambda$ of Definition \ref{def:infinite-path}.
  
  Conversely, if we have a representation $\pi$ of $C^*(\Lambda)$ on $L^2(\Lambda^\infty, \mu)$ with $\mu(Z(v)) < \infty$ for all $v\in \Lambda^0$, which arises from a $\Lambda$-projective system associated to the standard coding and prefixing maps $\sigma^n, \sigma_\lambda,$  then the representation is monic, and  $L^2(\Lambda^\infty,\mu)$ is isometric with $L^2(\Lambda^\infty,\mu_\pi)$. 
  \end{thm}

  \begin{proof} The proof is very similar to the proof of \cite[Theorem 4.2]{FGJKP-monic}, and because of that we will only give a quick sketch of it, which highlights the differences between the finite and row-finite cases. {The construction of a $\Lambda$-projective system on $\Lambda^\infty$ from a monic representation $\pi$ proceeds exactly as in \cite[Theorem 4.2]{FGJKP-monic}.} First, one defines the measure $\mu_\pi$ on cylinder sets  $Z(\lambda) \subseteq \Lambda^\infty$  by

  \[ \mu_\pi(Z(\lambda)) := \langle \xi, \pi(\chi_{Z(\lambda)}) \xi \rangle,\]
  where $\xi$ is the cyclic vector for $\pi$.  We use Carathe\'eodory's theorem to extend $\mu_\pi$ to a measure on $\Lambda^\infty$.
   Then the proof of \cite[Theorem 4.2]{FGJKP-monic} 
can be followed to construct a $\Lambda$-projective system $\{ f_\lambda\}_{\lambda \in \Lambda}$ associated to  the usual coding and prefixing maps $\{\sigma^n, \sigma_\lambda\}_{n, \lambda}$ on $\Lambda^\infty$, such that the affiliated $\Lambda$-projective representation is immediately seen to be unitarily equivalent to $\pi$.

 For the converse, suppose that we have a representation $\pi$ of $C^*(\Lambda)$ on $L^2(\Lambda^\infty, \mu)$ which arises from a $\Lambda$-projective system.  Then, in particular, 
 
 \[\pi(s_\lambda s_\lambda^*) = M_{\chi_{Z(\lambda)}}.\]
 
 Enumerate the vertices in $\Lambda^0$ as $\{ v_n\}_{n\in \N}$ and define 
\[\xi = \sum_{v_n \in \Lambda^0} \frac{\chi_{Z(v_n)}}{n \sqrt{\mu(Z(v_n))}} \in L^2(\Lambda^\infty, \mu).\]
(Note that $\mu(Z(v)) > 0$ for all $v$ by Condition (b) of Definition \ref{def-lambda-SBFS-1}.)
Then $\xi$ is cyclic for $C_0(\Lambda^\infty)$. To see this, it suffices to show that $\chi_{Z(\lambda)} \in \pi( C_0(\Lambda^\infty) )( \xi)$ for all $\lambda \in \Lambda$. Given $\lambda \in \Lambda$ with $r(\lambda) = v_n$, we compute:
\[ \pi (n \sqrt{\mu(Z(v_n))} \chi_{Z(\lambda)}) \xi = n \sqrt{\mu(Z(v_n))} \pi(s_\lambda s_\lambda^*) \xi = \chi_{Z(\lambda)},\]
since $Z(\lambda) \cap Z(v) \not= \emptyset$ only when $v = r(\lambda)$.
The fact that $L^2(\Lambda^\infty,\mu)$ is isometric with $L^2(\Lambda^\infty,\mu_\pi)$ now follows from the first part of the proof. 
  	\end{proof}

 \begin{rmk}
 \label{rmk:finite-vertex-condition}
{\color{white},}

\begin{enumerate}
\item  In the proof of the converse direction above, if $\mu(\Lambda^\infty) < \infty$ then we could alternatively take $\xi = \chi_{\Lambda^\infty}$ to be our cyclic vector.

\item 
  As shown in the proof above,  the measure $\mu_\pi$ on $\Lambda^\infty$ associated to a   monic representation $\pi: C^*(\Lambda) \to \mathcal{H}$ with monic vector $\xi$  must satisfy
\[
  \mu_\pi (Z(\lambda) ) =\langle \xi,  \pi(s_\lambda s_\lambda^*)\xi\rangle_{\mathcal{H}} \,  < \infty,\ \forall \lambda \in \Lambda.
  \]
 Additionally, $\mathcal{H}$ is isometric to $L^2(\Lambda^\infty, \mu_\pi)$ (see the proof of \cite[Theorem 4.2]{FGJKP-monic}).  If in particular  $\mathcal{H} = L^2(\Lambda^\infty,\mu) $, then  $L^2(\Lambda^\infty,\mu)$ and $L^2(\Lambda^\infty,\mu_\pi)$ are  isometric, and so    $ \mu (Z(\lambda) )=\mu_\pi (Z(\lambda) )  < \infty,$ $\forall \lambda \in \Lambda \iff \mu (Z(v) ) = \mu_\pi (Z(v) )  < \infty$. This shows that   $\mu (Z(v) ) = \mu_\pi (Z(v) )  < \infty$ is  a necessary condition for the monicity of a representation $\pi: C^*(\Lambda ) \to L^2(\Lambda^\infty,\mu) $.
 \end{enumerate}
 \end{rmk}

Next we extend to row-finite graphs \cite[Theorem 4.5]{FGJKP-monic}, which shows that a $\Lambda$-semibranching system on $(X, \mu)$ induces a monic representation of $C^*(\Lambda)$, with monic vector the characteristic function of the whole space, if and only if the associated range sets generate the $\sigma$-algebra  $\mathcal F$ of $X$. We also improve the aforementioned theorem by dropping the condition on the monic vector.   

The following Theorem is the only result in this paper which depends explicitly on the $\sigma$-algebra $\mathcal F$, so for this result only, we denote the associated Hilbert space by $L^2(X, \mathcal F, \mu)$.

\begin{thm}
\label{thm:range-sets-generate-sigma-alg-in-monic}
Let $\Lambda$ be a row-finite, source-free $k$-graph and let  $ \{t_\lambda\}_{\lambda \in \Lambda}$ be a $\Lambda$-semibranching  representation of $C^*(\Lambda)$ on $L^2(X, \mathcal{F}, \mu)$
with $\mu(X)< \infty$,  where $\mathcal{F}$ denotes the $\sigma$-algebra on $X$. Let $\mathcal{R} $ be the collection of sets which are modifications of range sets $ R_\lambda$ {by sets of measure zero}; that is, each element $Y \in  \mathcal R$ has the form 
	\[ Y = R_\lambda \cup S \quad \text{ or } \quad Y = R_\lambda \backslash S\]
	for some set $S$ of measure zero. 
	Let $\sigma(\mathcal{R})$ be the $\sigma $-algebra generated by $\mathcal{R}.$
The representation $ \{t_\lambda\}_{\lambda \in \Lambda}$ 
 is monic if and only if $\sigma(\mathcal R) = \mathcal F$. 
In particular, for a monic $\Lambda$-semibranching representation $\{t_\lambda\}_{\lambda \in \Lambda}$,  the set
	\[
	{\mathcal{S} }:=  \Big\{ \sum_{i=1}^{n} a_{i}t_{\lambda_i} t_{\lambda_i}^* \sum_k \frac{1}{k} \chi_{D_{v_k}} \ | \ n \in \N , \lambda_i \in \Lambda, a_i \in \C\Big\} = \Big\{ \sum_{i=1}^{n} a_{i} \chi_{R_{\lambda_i}}  \ | \ \ n \in \N , \lambda_i \in \Lambda, a_i \in \C\Big\}
	\]
	is dense in  $L^2(X, \mathcal F, \mu)$.
\end{thm}

\begin{proof}
Observe first that, if $\lambda \in v_k \Lambda$, then for any $\Lambda$-semibranching representation $\{ t_\lambda\}_{\lambda \in \Lambda}$, Remark \ref{rmk:Lambda-proj-is-mult} establishes that  
\[t_\lambda t_\lambda^* \sum_k \frac{1}{k} \chi_{D_{v_k}} =M_{\chi_{R_\lambda}} \sum_k \frac{1}{k} \chi_{D_{v_k}} =  \frac{1}{k} \chi_{R_{\lambda}}.\]
 In other words, the two descriptions of $\mathcal S$ given in the statement of the theorem are equivalent.

For the forward implication, suppose that the $
\Lambda$-semibranching
 representation $\{t_\lambda\}_{\lambda\in \Lambda}$ is monic and that $\xi$ is a monic vector for the representation. By Proposition~\ref{supportcyclic}, the support of $\xi$ is equal to $X$ a.e.. Moreover, Remark \ref{rmk:Lambda-proj-is-mult} establishes  that 
 \[ M:=\text{span} \{ t_\lambda t_\lambda^* (\xi): \lambda \in \Lambda\} =\text{span} \{ \chi_{R_\lambda} \xi: \lambda \in \Lambda \},\]
which is dense in $L^2(X, \mathcal F, \mu)$.  Furthermore,  the support of any function in $M$ belongs to $\sigma(\mathcal{R})$ since the support of $\xi$ is $X$ a.e.. Therefore for any $f \in L^2(X, \mathcal F, \mu)$ there is a sequence $(f_j)_j$, with{ $f_j \in M$}, such that 
\[\lim_{j\to \infty}  \int_X | f_j - f|^2 \, d\mu =0.\]
In particular, $(f_j) \to f$ in measure. 
The rest of the proof follows now exactly as in the proof of Theorem 4.5 in \cite{FGJKP-monic}.

For the converse, suppose $\sigma(\mathcal R) = \mathcal F$; we will show that {$\chi_X$ is a monic vector for the representation by showing that $\mathcal S$ is dense in $L^2(X, \mathcal F, \mu)$}.  Fix $f \in L^2(X,  \sigma(\mathcal R), \mu)$ and $\epsilon > 0$.  Choose a simple function $\phi = \sum_{i=1}^n a_i \chi_{X_i}$, with $X_i \in \sigma(\mathcal R)$, such that $\int_X |\phi -f|^2 d\mu < \epsilon$, and set $A : = \max  | a_i|$. {By \cite[Lemma A.2.1(ii)]{BDurrett}, for each $i$, there is an element $B_i$ of the algebra $\widetilde {\mathcal R}$ of sets generated by $\mathcal R$ such that $\mu(B_i \Delta X_i) < \frac{\epsilon}{n(n-1) A^2}$. That is, each $B_i$ is a finite union of elements of $\mathcal R$, so for each $i$ there exists a finite collection  $\{ R_{\lambda_{ij}}\}_j$ of range sets such that $\mu\left(B_i \Delta (\cup_j R_{\lambda_{ij}}) \right) =0.$  In other words, if we  define $\psi = \sum_{i=1}^n a_i \chi_{B_i}$, then $\psi \in \mathcal S$. Moreover, 
\begin{align*}  \int_X |\psi - \phi|^2 \, d\mu & = \int_X \left| \sum_i a_i (\chi_{B_i \backslash X_i} - \chi_{X_i \backslash B_i} ) \right| ^2\, d\mu \leq \int \left( \sum_i |a_i| \chi_{X_i \Delta B_i} \right)^2 \, d\mu \\
&  = \sum_i \int_{X_i \Delta B_i} |a_i|^2 \, d\mu + 2 \sum_{i\not= j} |a_i | \, |a_j| \mu ( (X_i \Delta B_i) \cap (X_j \Delta B_j) )  < 2\epsilon.
\end{align*}
The Cauchy-Schwarz inequality now implies that 
\begin{align*}
 \int_X | \psi - f|^2 \, d\mu & \leq \int_X (|\psi - \phi| + | \phi - f|)^2 \, d\mu \\
 & \leq \int_X |\psi - \phi|^2 \, d\mu + \int_X |\phi - f|^2 \, d\mu + 2 \int_X | \phi - \psi| \, | \phi- f| \, d\mu \\
 &
  < 2 \epsilon + \epsilon + 2\sqrt 2 \epsilon.
 \end{align*}
It follows that $\mathcal S$ is dense in $L^2(X, \mathcal F, \mu)$ as claimed.}
\end{proof}

We conclude this section with a necessary condition for monicity of a representation of a $k$-graph $C^*$-algebra.

 \begin{prop}
\label{prop:vertex-prop-implies-no-monic-repns}
Let $\Lambda$ be a row-finite source-free  $k$-graph and fix a $\Lambda$-semibranching function  system on a $\sigma$-finite measure space $(X, \mu)$.  Suppose that a vertex $v$ of $\Lambda$ satisfies 
\begin{equation}\label{eq:eq-needed-cycle-no-entry} \forall \tau \in v\Lambda:\  R_{\tau}=R_{v\tau} =D_v.  \end{equation}
If there is a measurable subset $X_v$ of $D_v$ with $0 < \mu(X_v) < \mu(D_v)$, then no   $\Lambda$-projective representation arising from this $\Lambda$-semibranching function system will be monic.
\end{prop}
\begin{proof}
Let $\{t_\lambda\}_{\lambda \in \Lambda}$ be a $\Lambda$-projective representation of $C^*(\Lambda)$ on $L^2(X, \mu)$, and write $\pi$ for the induced representation of $C_0(\Lambda^\infty)$. Recall from Remark \ref{rmk:Lambda-proj-is-mult} that 
\[ t_\lambda t_\lambda^*(f) = \pi(\chi_{Z(\lambda)})(f) = \chi_{R_\lambda} f\]
for any finite path $\lambda \in \Lambda$ and any $f \in L^2(X, \mu)$.

Suppose that $v$  and  $X_v \subseteq D_v$ are as in the statement of the theorem.  We will show that no vector $\xi \in L^2(X, \mu)$ can be monic. 
Recall from Proposition \ref{supportcyclic} that if $\xi$ is monic for $\pi$ then $X \backslash \text{supp}(\xi)$ has measure zero.  Choose, therefore, $\xi \in L^2(X, \mu)$ with a.e.~full support.  Let 
\[M = \overline{\text{span}} \{ \chi_{R_\lambda} \xi: \lambda \in \Lambda \} = \overline{\text{span}} \{t_\lambda t_\lambda^*(\xi)\}  = \pi (C_0(\Lambda^\infty))(\xi).\] 
Equation \eqref{eq:eq-needed-cycle-no-entry} implies that, if $r(\lambda) = v$, then 
 $\chi_{R_\lambda} \xi|_{D_v} = \xi|_{D_v}$. 
  So if $f\in M$ then the support of
$f|_{D_v}$ is either $D_v$ or has measure zero. 
Therefore $\chi_{X_v} \notin M$ and hence $\xi$ is not monic. 
\end{proof}

\begin{rmk}By Definition \ref{def-lambda-SBFS-1}, if we have a $\Lambda$-semibranching function system on $(X, \mu)$, then  for every fixed $m \in \N^k$,
\[
D_v=\bigcup_{\lambda \in v \Lambda^m} R_{\lambda}, 
\]
and each $R_\lambda$ has positive measure.  It follows that 
 Equation \eqref{eq:eq-needed-cycle-no-entry} is satisfied iff
\begin{equation}\label{eq:conditon-for-no-existence-set}
\hbox{For  every $m \in \N^k$ there is precisely one path of degree $m$ having range $v$.}
\end{equation}
\end{rmk}

\begin{defn}
\label{def:cycle-without-entr}
In a directed graph (1-graph) $E$, a {\em cycle} is a finite path  $e_1 e_2 \cdots e_n$ with $s(e_n) = r(e_1)$.  We say that a cycle {\em has an entrance} if there is an $i \leq n$ and an edge $e$ with $r(e) = s(e_i)$ but $e \not= e_{i+1}$.
\end{defn}

\begin{cor}
\label{prop:no-entrance-implies-no-monic-repns}
Let $E$ be a row-finite source-free directed graph which contains a cycle without entrance, and let $(X, \mu)$ be $\sigma$-finite.  If a $\Lambda$-projective representation of $C^*(E)$ on $L^2(X, \mu)$ is monic, then for every vertex $v$ lying on a cycle without entrance, $D_v$ has no measurable subsets $X_v$ with $0 < \mu(X_v) < \mu(D_v)$.
\end{cor}
\begin{proof}
By definition, if $v$ lies on a cycle without entrance, then $v$ satisfies Equation \eqref{eq:conditon-for-no-existence-set}. The result now follows from Proposition \ref{prop:vertex-prop-implies-no-monic-repns}.
\end{proof}

{
\begin{rmk} In \cite[Example~4.7]{FGJKP-monic}, the authors present an example of a $\Lambda$-semibranching representation associated to a 1-graph, on $L^2([0,1])$ (with usual Lebesgue measure), such that the vector $\chi_{[0,1]}$ is not monic. Indeed, in that graph, there is a vertex $v = v_2$ which supports a loop with no entrance, such that $D_v = (1/2, 1]$.  Thus,  the corollary above shows that the representation constructed in \cite[Example~4.7]{FGJKP-monic} cannot be monic, and has no monic vectors.
\end{rmk}
}

\begin{rmk}
\label{rmk:higher-rank-cycles-w-o-entrance}
Although Evans and Sims generalized  the concept of ``cycle without entrance'' to higher-rank graphs in \cite{evans-sims} (see \cite[Remark 3.6]{evans-sims}), we do not have an analogue of Corollary \ref{prop:no-entrance-implies-no-monic-repns} for higher-rank cycles without entrance: the 2-graph of  \cite[Example 6.1]{evans-sims} contains a generalized cycle without entrance but does not satisfy Equation \eqref{eq:conditon-for-no-existence-set}.  The cycline pairs introduced in \cite[Definition 4.3]{BNR} can also be viewed as a  generalization of the notion of ``cycles without entrance'' to higher-rank graphs -- by \cite[Remark 4.9]{BNR}, a generalized cycle without entrance is a cycline pair if every extension of that generalized cycle without entrance is also a generalized cycle without entrance.  However, even this stronger condition is not enough to guarantee Equation \eqref{eq:conditon-for-no-existence-set}.  Section 5.11 of \cite{mcnamara} (cf.~also \cite[Section 5.1]{FGLP}) exhibits a 2-graph with one vertex for which Equation \eqref{eq:conditon-for-no-existence-set} fails, but which has nontrivial cycline pairs such as  $(e_1, f_1)$.
\end{rmk}

\section{Irreducibility} 
\label{sec:irred}

In this section, we describe a variety of necessary and sufficient conditions for a $\Lambda$-projective representation to be irreducible. 
{We begin by defining the notions of ergodicity that we will use to analyze the irreducibility of $\Lambda$-projective representations, and describing them in both measure- and operator-theoretic terms.  Proposition \ref{prop:irred-implies-ergodic} then identifies the commutant of a $\Lambda$-projective representation on $(X, \mu)$ in terms of invariant functions on $X$.

Section \ref{sec:necessary} highlights three necessary conditions for a $\Lambda$-projective representation to be irreducible, but shows by example that these conditions are not sufficient.  However, in Section~\ref{sec:sufficient} we identify a variety of conditions under which a $\Lambda$-projective representation on the infinite path space $\Lambda^\infty$ must be irreducible.  The final subsection, Section~\ref{sec:random}, includes a variety of other results related to the irreducibility of $\Lambda$-semibranching representations.}

\begin{defn}\label{def:ergodic-maps}
Let $(X, \mu)$ be a measure space. A function $T: X \to X$ is {\em ergodic} with respect to $\mu$ if whenever $\mu(A \Delta T^{-1}(A)) = 0$ (i.e., $A$ is invariant {as in Definition \ref{def:invariant-set}}) we have either $\mu(A) = 0$ or $\mu(X \backslash A) = 0$. A family of maps $\{ T_i\}_{i\in I} $ is {\em jointly ergodic} with respect to $\mu$ if whenever $\mu(A \Delta T_i^{-1}(A)) = 0 $ for all $i$, we have either $\mu(A) = 0$ or $\mu(X \backslash A) = 0$.
\end{defn}

\begin{rmk}
If $\mu(X) < \infty$, this definition agrees with the definition of ergodicity used in earlier papers such as \cite{FGJKP-monic}.
\end{rmk}

\begin{defn}
\label{def:invariant-fcn}
Let $(X, \mu)$ be a measure space, and $T: X \to X$ a function. A measurable function $f$ is \emph{invariant} with respect to $T$ if $f\circ T = f $ a.e.. {We say that $f$ is \emph{jointly invariant} with respect to a family of maps $\{T_i\}_{i\in I}$ if $f\circ T_i = f$ a.e..} for all i.
\end{defn}

The following lemma is undoubtedly well-known to experts; we include it here for completeness and ease of reference.

\begin{lemma}\label{lem:ergodic-invariant} Let $(X, \mu)$ be a measure space, and fix $T: X \to X$ and $T_i:X \to X$, $i\in I$. Then,
\begin{itemize}
 \item[(a)]  $T$ is ergodic if, and only if, every invariant measurable function $f: X \to \C$ is constant a.e..
\item[(b)]  {The family $\{T_i\}$ is jointly ergodic if, and only if, every jointly invariant measurable function $f: X \to \C$ is constant a.e..}
\end{itemize}
\end{lemma}
\begin{proof}

{We prove (a) and leave the proof of (b) to the reader, as it is essentially the same.}

Suppose first that every invariant measurable function on $X$ is constant. Choose $A \subseteq X$ such that $\mu(A \Delta T^{-1}(A)) = 0$.  Then $\chi_A \circ T = \chi_A$ a.e. and so $\chi_A$ is constant a.e.. Therefore either $\chi_A = 0$ a.e.~(and so $\mu(A) = 0$) or $\chi_A = 1$ a.e., in which case $\mu(A^c) = 0$.

Suppose now that $T$ is ergodic and let $f$ be a measurable invariant function. For each $r\in \R$, set 
\[ I_{r}=\{x\in X: \text{Im}\,f(x)>r\}, \quad R_r = \{ x\in X: \text{Re}\, f(x) > r\}, \quad E_r \in \{ I_r, R_r\}.\]
Notice that each $E_r$, for $r \in \R$, is measurable and invariant with respect to $T$, and that if $s > r$ then $I_s \subseteq I_r$ and $R_s \subseteq R_r$. Since $T$ is ergodic, we obtain that, for all $r$,
$\mu(E_r)=0$ or $\mu(X\setminus E_r)=0$. 
Before we proceed, we quickly observe that a real-valued function $f$ is constant if, and only if, there exists $C$ such that for all $r\geq C$,  $\mu(R_r)=0$, and for all $r<C$, $\mu(X\setminus R_r)=0$.

Now, let $K_I=\sup\{r:\mu(I_r)\neq 0\}$ and $K_R = \sup\{ r: \mu(R_r) \not= 0\}$. 
First, we will show that $ K_I < \infty$; a similar argument will also establish that $K_R < \infty$.  If $x\in
\bigcap_{n\in \N} I_n$ then $f(x)=+\infty$. Consequently, $\mu
(\bigcap_{n\in\N}I_n)=0$ (in fact $\bigcap_{n\in\N}I_n = \emptyset$).
As $f$ is invariant and $T$ is ergodic,  for every $r$, either $\mu(I_r)=0$ or $\mu(X\setminus I_r)=0$.
Notice that if $\mu(I_r)=0$ for some $r$ then $\mu(I_s)=0$ for all
$s>r$ and in this case $K_I$ is finite. Suppose that $\mu(X\setminus
I_r)=0$ for all $r$. Then $\mu \left(X\setminus \bigcap_{n\in \N} I_n
\right)=\mu \left( \bigcup_{n\in \N} X\setminus I_n \right) =  0$,
which implies that $\mu(X)=0$, a contradiction.

  Moreover, $\mu(I_{K_I})=0$, since $I_{K_I}=\bigcup_{n\in \N} I_{K_I+\frac{1}{n}}$ and $\mu(I_{K_I+\frac{1}{n}})=0$ for all $n$. We now use the  characterization of a constant function given above. For $r\geq K_I$ we have that $\mu(I_r)\leq \mu (I_{K_I})=0$ and, for $r<K_I$  , there exists $r_1$ such that $r<r_1<K_I$ and $\mu(X\setminus I_{r_1})=0$ (from the definition of sup and the fact that the invariant set $I_{r_1}$ satisfies $\mu(I_{r_1})\neq 0$). Since $I_{r_1}\subseteq I_r$, this implies that $\mu(X\setminus I_r)=0$.  In other words, whenever $ r < K_I$, there is  a set of full measure on which $\text{Im}f(x) > r$.  It  follows that $\text{Im}(f) = K_I$ a.e..
A similar argument will show that $\text{Re}(f) = K_R$ a.e., completing the proof that $f = K_R +  iK_I$ is constant a.e..
\end{proof}

{One of the equivalent ways to describe irreducibility of a representation is via its commutant. Therefore, we give a description of the commutant of a representation associated to a $\Lambda$-projective system below. }

 \begin{prop}
 \label{prop:irred-implies-ergodic}
    Let $\Lambda$ be a row-finite $k$-graph with no sources.  
  Suppose that we have a $\Lambda$-projective system on a locally compact, $\sigma$-compact, Hausdorff space $(X, \mu)$, where $\mu$ is a Radon measure, and let 
 $\{T_\lambda:\lambda\in\Lambda\}$  be the associated representation of $C^*(\Lambda)$ on $L^2(X, \mu)$.
  Then, the commutant of the operators $\{T_\lambda: \lambda \in \Lambda\}$ consists of multiplication operators, and contains the set of multiplication
      operators by functions $h$ with $h \circ \tau^n = h$ $\mu$-a.e., for all $n\in \N^k$.  
\end{prop}

 \begin{proof}
 Let $T \in B(L^2(X, \mu))$ be an operator in the commutant of the $\Lambda$-projective representation. Then $T$ commutes with the representation $\pi$
 of the abelian subalgebra $C_0(\Lambda^\infty)$ of $C^*(\Lambda)$, where 
 \[ \pi(\chi_{Z(\lambda)}) = T_\lambda T_\lambda^* = M_{\chi_{R_\lambda}}.\]
By \cite[Theorem 6.3.4 and Proposition 6.3.6]{arveson}, the fact that $\mu$ is a Radon measure on the locally compact, $\sigma$-compact, Hausdorff space $X$, enables us to invoke the proof of \cite[Proposition 4.7.6]{pedersen} to conclude that $T = M_h$, for some $h \in L^\infty(X)$. It is easy to check that any  multiplication operator $T = M_h$, where $h \circ \tau^n = h$ a.e. for all $n\in \N^k$, commutes with each operator $T_\lambda$.

\end{proof}

\subsection{Necessary conditions for irreducibility}
\label{sec:necessary}
{In this subsection we describe three necessary conditions for the irreducibility of a $\Lambda$-projective representation of $C^*(\Lambda)$ arising from $\Lambda$-projective systems. Proposition \ref{ventosul} tells us that all irreducible $\Lambda$-projective representations arise from $\Lambda$-semibranching function systems with jointly ergodic coding maps.  From this, we deduce Theorem~\ref{prop:cofinal}, {which states} that only cofinal $k$-graphs admit irreducible $\Lambda$-semibranching representations. Finally, Proposition \ref{prop:irred-implies}  describes how the functions $f_\lambda$ associated to an irreducible $\Lambda$-projective representation must interact.  We also present a variety of examples, indicating the insufficiency of these necessary conditions to guarantee irreducibility.}

\begin{prop}
 \label{ventosul}
    Let $\Lambda$ be a row-finite $k$-graph with no sources. Suppose that we have a $\Lambda$-projective system on $(X, \mu)$, and let 
 $\{T_\lambda:\lambda\in\Lambda\}$  be the associated representation of $C^*(\Lambda)$ on $L^2(X, \mu)$. If the representation $\{T_\lambda:\lambda\in\Lambda\}$ is
      irreducible, then the coding maps $\tau^n$ are jointly ergodic with respect to the measure $\mu$.    
     \end{prop}
\begin{proof}

Suppose  the representation generated by $\{T_\lambda\}_\lambda$ is irreducible, and suppose  $A \subseteq X$  satisfies $\mu( A \Delta (\tau^{n})^{-1}(A) ) = 0$ for all $n$.  Then both $(\tau^n)^{-1}(A) \backslash A$ and $A \backslash (\tau^n)^{-1}(A)$ have measure zero, for any $n\in \N^k$, and consequently 
 \[ M_{\chi_A \circ \tau^n } = M_{\chi_{(\tau^{n})^{-1}(A) } }=  M_{\chi_{(\tau^n)^{-1}(A) \cap A} }= M_{\chi_A}.\]
 
 It now follows from the definition of the operators $T_\lambda$ that, for all $\lambda \in \Lambda$, $M_{\chi_A}$ commutes with $T_\lambda$ and $T_\lambda^*$. {To see this, recall that 
 \[ T_\lambda M_{\chi_A}(f) = f_\lambda \cdot (\chi_A \circ \tau^{d(\lambda)}) \cdot (f \circ \tau^{d(\lambda)})\]
 while $M_{\chi_A} T_\lambda(f) = \chi_A \cdot f_\lambda \cdot (f \circ \tau^{d(\lambda)})$, so since $f_\lambda$ is supported on $R_\lambda$ it suffices to show that $\chi_{A \cap R_\lambda}  = \chi_{R_\lambda}( \chi_A \circ \tau^{d(\lambda)}) $.  Observe that 
 \[ \chi_{R_\lambda} (\chi_A \circ \tau^{d(\lambda)}) = \chi_{\tau_\lambda(A)},\]
 and consequently it suffices to show that $\mu((A \cap R_\lambda) \Delta \tau_\lambda(A)) = 0$.  
 
 By definition, $(A \cap R_\lambda) \Delta \tau_\lambda(A) = \{ \tau_\lambda (z) \in A : z \not\in A\} \cup \{ \tau_\lambda( z) \not \in A: z \in A\}$.  Our hypothesis that $A$ is invariant implies that 
 \begin{align*}
 0 = \mu(A \Delta (\tau^{d(\lambda)})^{-1}(A)) &  = \mu \left( \{ y \in A: y \not= \tau_\eta( z) \text{ for any } z \in A , d(\eta) = d(\lambda) \} \cup \{ y \not\in A: \tau^{d(\lambda)}(y) \in A\}\right)
 \end{align*}
 Restricting the sets described in the previous equation to $R_\lambda$ will not increase their measure, so we also have  
 \begin{align*}
 0 &=   \mu\left( \{ y \in A \cap R_\lambda: y \not= \tau_\lambda(z) \text{ for  any }z \in A\} \cup \{ y \in R_\lambda \backslash A : \tau^{d(\lambda)}(y) \in  A\} \right) \\
 &= \mu \left( \{ \tau_\lambda(z) \in A: z \not\in A\} \cup \{ \tau_\lambda(z) \not\in A : z \in A\} \right)
 \end{align*}
 by using the fact that $\tau^{d(\lambda)} \circ \tau_\lambda = id|_{D_{s(\lambda)}}$ a.e.~and hence $\tau_\lambda$ is injective a.e..   We conclude that $T_\lambda$ commutes with $M_{\chi_A}$ as claimed.
 
 To see that $T_\lambda^*$ also commutes with $M_{\chi_A}$, we use the fact that $T_\lambda M_{\chi_A} = M_{\chi_A} T_\lambda$.  Thus, for any $f, g \in L^2(X, \mu)$,
 \[ \langle T_\lambda M_{\chi_A} f, g \rangle = \langle M_{\chi_A} T_\lambda f, g \rangle = \langle f, T_\lambda^* M_{\chi_A} g\rangle
\] 
as $M_{\chi_A}$ is self-adjoint.  On the other hand, we also have  $\langle T_\lambda M_{\chi_A}  f, g \rangle = \langle f,  M_{\chi_A} T_\lambda^* g \rangle,$ so for all $f, g \in L^2(X, \mu),$
\[ \langle f, (T_\lambda^* M_{\chi_A} - M_{\chi_A} T_\lambda^*) g\rangle = 0.\]
It follows that $T_\lambda^* M_{\chi_A} - M_{\chi_A} T_\lambda^* = 0$, so $M_{\chi_A}$ commutes with $T_\lambda^*$ as well, for all $\lambda \in \Lambda$.
 }

  As the representation is irreducible, this implies that $M_{\chi_A} \in \C 1 = \C M_{\chi_{X}}$.
  Since $M_{\chi_A}$ is a projection, it follows that either $\mu(A) = 0$ (if $M_{\chi_A} = 0$) or $\mu(A^c) = 0$ (if $M_{\chi_A} = 1 = 
  M_{\chi_{X}}$).
 Thus, the operators $\tau^n$ are jointly ergodic with respect to the measure $\mu$.
 \end{proof}
 
 Next, we give an example that the converse of the above proposition does not hold in general.

\begin{example}\label{seal} Let $E$ be the graph with one vertex $v$ and one edge $e$. We will build a semibranching function system for $E$ on the set $\{0,1\}$ with counting measure.
Let $R_{e}=\{0,1\}=D_v$ and define the prefixing map $f_e:D_{v} \rightarrow R_e$ by $f_e(x)=x+1 \mod 2$. Then $\{D_{v},R_{e},f_e\}$ is a branching system in the sense of \cite[Definition~3.6]{goncalves-li-royer-SBFS}. By \cite[Proposition~3.12 and Remark~3.11]{goncalves-li-royer-SBFS}, this branching system gives rise to a semibranching function system. 
Let $\pi$ be the induced representation of $C^*(\Lambda)$ on $\ell^2(\{ 0,1\})$. 
By Proposition~\ref{prop:vertex-prop-implies-no-monic-repns} this representation is not monic. Hence, Theorem~\ref{vinhosoon} implies that $\pi$ is not irreducible.  Indeed $\ell^2\{0\}$ and $\ell^2\{1\}$ are invariant subspaces of $\ell^2\{0,1\}$. On the other hand, there are no nonempty proper subsets of $\{0,1\}$ which are invariant by all of the powers of the coding map.
\end{example}

\begin{thm} 
	\label{prop:cofinal}
	Let $\Lambda$ be a row-finite $k$-graph with no sources. Suppose $\Lambda$ is not cofinal (Definition \ref{def:cofinal}).
Then there are no irreducible representations of $C^*(\Lambda)$ arising from $\Lambda$-semibranching function systems.
\end{thm}
\begin{proof}
{Let  $x \in \Lambda^\infty$ and $v \in \Lambda^0$ be obtained from the fact that $\Lambda$ is not cofinal.}
Given a $\Lambda$-semibranching function system $\{ \tau^n, \tau_\lambda\}$ on $(X, \mu)$, let $A$ denote the smallest  set which  contains  $R_{r(x)}$ and is invariant with respect to all the coding maps $\tau^n$.  Observe that $\mu(A \cap R_v) = 0$ so, since $ \mu(R_v) > 0$ and $\mu(A) \geq \mu(R_{r(x)}) > 0$, it follows from Proposition~\ref{ventosul}  that any representation associated to a $\Lambda$-semibranching function system will be reducible.
\end{proof}

{\begin{rmk}
However, $\Lambda$-semibranching function systems on non-cofinal $k$-graphs can give rise to irreducible representations if we restrict our attention to a  minimal $\mu$-invariant subset; see Definition~\ref{def:in-min-inv-subsets} and Theorem~\ref{thm:ergodic-inf-path-space} below.
\end{rmk}}

We now conclude this subsection by applying the work of Carlsen et al.~\cite{CKSS} on irreducible representations of $k$-graphs to the setting of $\Lambda$-projective representations. {We remark that when the following definitions occur in Carlsen et al Section 4, they assume $\Lambda^0$ is a maximal tail; we have not invoked this hypothesis as it merely implies that the relation $\lambda \sim \nu \Leftrightarrow (\lambda, \nu) \text{ periodic}$ is an equivalence relation, which is irrelevant for our purposes.}

\begin{defn}
Let $\Lambda$ be a $k$-graph. A pair $(\lambda, \nu) \in \Lambda \times \Lambda$ is a {\em periodic pair} if $s(\lambda) = s(\nu)$ and  $\lambda x = \nu x$ for all $x \in Z(s(\lambda))$.  We set 
\[ \text{Per}(\Lambda) := \{ d(\lambda) - d(\nu): (\lambda, \nu) \text{ periodic}\}.\]
One can check that $\text{Per}(\Lambda)$ is a subgroup of $\Z^k$.

Following \cite{CKSS}, we set $H_{\text{Per}}(\Lambda) $ to be the set of vertices which realize all of $\text{Per}(\Lambda)$: that is, 

\[ H_{\text{Per}}(\Lambda) = \{ v \in \Lambda^0: \text{ if } r(\lambda) = v \text{ and { $\exists m \in \N^k$: }} d(\lambda) - m \in \text{Per}(\Lambda), \text{ then } \exists \mu \in v\Lambda^m \text{ s.t. } (\lambda, \mu) \text{ periodic}\}.\]
\end{defn}
Observe that if $(\lambda, \nu)$ is a periodic pair then $\sigma^{d(\lambda)} = \sigma^{d(\nu)}$ on $Z(\lambda) = Z(\nu)$.  However, we need not have $\tau^{d(\lambda)} = \tau^{d(\mu)}$ on an arbitrary $\Lambda$-semibranching function system.

\begin{example}
\label{ex:periodic-but-not-coding}
Consider the following 1-graph $\Lambda$:
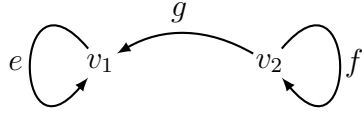
\[
\begin{tikzpicture}[scale=1.5]
 \node[inner sep=0.5pt, circle] (v) at (0,0) {$v_1$};
    \node[inner sep=0.5pt, circle] (w) at (1.5,0) {$v_2$};
    \draw[-latex, thick] (w) edge [out=50, in=-50, loop, min distance=30, looseness=2.5] (w);
    \draw[-latex, thick] (v) edge [out=130, in=230, loop, min distance=30, looseness=2.5] (v);
\draw[-latex, thick] (w) edge [out=150, in=30] (v);
\node at (-0.75, 0) {\color{black} $e$}; 
\node at (0.7, 0.45) {\color{black} $g$};
\node at (2.25, 0) {\color{black} $f$};
\end{tikzpicture}
\]
Define a $\Lambda$-semibranching function system on $[0, 1]$ (equipped with Lebesgue measure) by setting $D_{v_1} = (0, 1/2)$ and $D_{v_2} = (1/2, 1)$, and defining
\[ \tau_e(x) = \tau_g(x) = \frac{x}{2}, \qquad \tau_f(x) = \frac{3}{2} - x.\]
The coding map is then given by $\tau^1(x) = \begin{cases}
2x, & x \in (0, 1/2)\\
\frac{3}{2} - x, & x \in (1/2, 1).
\end{cases}$

One easily checks that 
the hypotheses of a $\Lambda$-semibranching function system are satisfied.  Moreover, $(f, v_2)$ is a periodic pair, but $\tau^1 \not= id$ on $R_{f} \cap R_{v_2} = (1/2, 1)$.
\end{example}

\begin{prop}
\label{prop:periodic-implies-same-meas}
In any $\Lambda$-semibranching function system, if $(\lambda, \nu)$ is a periodic pair, then $\mu (R_\lambda \Delta R_\nu) = 0.$ 
\end{prop}	

\begin{proof}
Suppose $x \in R_\lambda$.  By definition, $x = \tau_\lambda(y)$ for some $y \in D_{s(\lambda)} = D_{s(\nu)}$.  Let $m = \left( d(\lambda) \vee d(\nu)  \right)- d(\lambda)$; by the definition of a $\Lambda$-semibranching function system, we can write 
\[ D_{s(\lambda)} = \bigsqcup_{\eta \in s(\lambda)\Lambda^m} R_\eta\]
modulo sets of measure zero.  Therefore, without loss of generality, we may assume $y = \tau_\eta(z)$ for some $\eta \in s(\lambda)\Lambda^m$ and some $z \in D_{s(\eta)}$.  

The fact that $(\lambda, \nu)$ is a periodic pair implies, in particular, that  since $d(\lambda \eta)  = d(\lambda ) \vee d(\nu) \geq d(\nu)$, we have $\lambda \eta = \nu \tilde \eta$ for some $\tilde \eta \in s(\nu) \Lambda$.  Consequently, since $x = \tau_\lambda(y) = \tau_{\lambda \eta}(z)$, we must have $x = \tau_{\nu \tilde \eta}(\tilde z)$ for some $\tilde z \in R_{\tilde \eta}$.  It follows that $x = \tau_\nu (\tau_{\tilde \eta}(\tilde z)) \in R_\nu$, and  so almost all $x \in R_\lambda$ are also in $R_\nu$, as claimed.
\end{proof}

\begin{prop}
Let $\Lambda$ be a row-finite, source-free $k$-graph and let $\{T_\lambda\}_{\lambda \in \Lambda}$ be a $\Lambda$-projective representation of $C^*(\Lambda)$ on $L^2(X, \mu)$, with associated functions $f_\lambda$ defined on $R_\lambda$.  If $\{T_\lambda\}_{\lambda \in \Lambda}$ is irreducible {and for every periodic pair $(\lambda, \nu)$ we have $\tau^{d(\lambda)} = \tau^{d(\nu)}$ on $R_\lambda \cap R_\nu$}, then there exists $z \in \T^k$ such that  for every periodic pair $(\lambda, \nu)$ with $r(\lambda) = r(\nu) \in H_{Per}(\Lambda)$, we have 
\[ \frac{f_\lambda}{f_\nu} = z^{d(\lambda) - d(\nu)} \text{ on } R_\lambda \cap R_\nu.\]
\label{prop:irred-implies}
\end{prop}
\begin{proof}
Suppose  $\{T_\lambda\}_\lambda$ is irreducible.  As $T_v = M_{\chi_{R_v}}$ is nonzero for each $v \in \Lambda^0,$, the maximal tail $T$ invoked in  \cite[Theorem 5.3(2)]{CKSS} is given in our case by $T = \Lambda^0$.  Therefore, \cite[Theorem 5.3(2)]{CKSS} implies that there exists $z \in \T^k$ (which implements a character of the periodicity group $\text{Per}( \Lambda)$) such that  $T_\lambda = z^{d(\lambda) - d(\nu)} T_\nu$ for all $(\lambda, \nu)$ as in the statement of the proposition.
Moreover, if $(\lambda, \nu) $ is periodic, then by  Proposition \ref{prop:periodic-implies-same-meas}, $R_\lambda = R_\nu$ a.e.~and therefore $f_\lambda, f_\nu$ are defined and nonzero on $R_\lambda \cap R_\nu$.  
From our hypothesis that $\tau^{d(\lambda)} = \tau^{d(\nu)}$ on this domain, it follows that
\[T_\lambda = z^{d(\lambda) - d(\nu)} T_\nu \Leftrightarrow \frac{f_\lambda}{f_\nu} = z^{d(\lambda) - d(\nu)} \text{ on } R_\Lambda \cap R_\nu. \qedhere\]
\end{proof}

However, the converse of this proposition fails.
In the discussion of the following example, we will interpret the 0th power of a loop $\lambda$ to mean the vertex $r(\lambda)$.
\begin{example}
\label{ex:CKSS}
Consider the 1-graph $\Lambda$ from Example \ref{ex:periodic-but-not-coding}.
Observe that in this case $\Lambda^\infty = \{ e^\infty, \{e^n g f^\infty\}_{n \in \N}, f^\infty\}$, and that for all $m, n\in \N$,
\[ Z(f) = Z(f^m) = \{ f^\infty\};  \quad Z(e^n) = \{ e^\infty , \{ e^m g f^\infty\}_{m \geq n} \}; \quad  Z(e^n g) = Z(e^n g f^m) = \{ e^n g f^\infty\}.\]
Therefore, every infinite path save for $e^\infty$ constitutes an clopen set, and so the topology and the Borel $\sigma$-algebra of $\Lambda^\infty$ is precisely the power set of $\Lambda^\infty$.  Moreover, the set of periodic pairs is $\{(g^i f^m, g^i f^n): i \in \{ 0, 1\}, m, n\in \N\},$ so $\text{Per}(\Lambda) = \Z$.
However, the set $H_{\text{Per}}(\Lambda)$ of vertices with maximal periodicity is $\{ v_2\}$, since the paths $e^n$ all have source $v_1$ but do not appear in any periodic pair.

Define a measure $\mu$ on $\Lambda^\infty$ by  
%{\color{red}Note that this isn't a probability measure; the measure of the total space is $3/4$.}
\[ \mu (Z(f)) = 1/4; \qquad \mu(Z(e^n g))=\frac{1}{2^{n+2}}; \qquad \mu(Z(e^n))  = \frac{1}{4} + \frac{1}{2^{n+1}}.\]
Observe that this measure gives $\mu(\{e^\infty\}) = 1/4$.  Furthermore, $(\sigma^n)^{-1}(\{e^\infty\}) = \{e^\infty\}$ for all $n$, so $\mu(\sigma^{-1}(\{e^\infty\}) \Delta \{ e^\infty\}) = 0$, but neither $\{e^\infty\}$ nor $\Lambda^\infty \backslash \{ e^\infty\}$ has measure zero.  
  Theorem 3.12(b) of \cite{FGJKP-monic} (or Proposition \ref{ventosul} above) now implies that any 
  $\Lambda$-projective representation associated to $(\Lambda^\infty, \mu)$ will fail to be irreducible. 

We now compute the Radon--Nikodym derivatives associated to   the $\Lambda$-semibranching function system on $(\Lambda^\infty, \mu)$ given by the usual coding and prefixing maps.  First, the fact that the cylinder sets $Z(g) = Z(g f^m)$ and $Z(f) = Z(f^m)$ consist of single points makes it easy to compute that
\[ \Phi_g = \Phi_{g f^m} = 1 \ \text{ for all } m \in \N; \qquad \Phi_f = \Phi_{f^m} = 1 \ \forall \ m \in \N.\]
However, $\Phi_e$ is not constant on its domain $Z(v_1)$: 
\[ \Phi_e(e^\infty) = \frac{d(\mu \circ \sigma_e)}{d\mu}   (e^\infty) =1, \quad \text{ but } \Phi_e(e^ngf^\infty) = \frac{1}{2}.\]
 
 In particular, if (for each finite path $\lambda$) we set $f_\lambda = \Phi_\lambda^{-1/2}$ on its domain $Z(\lambda)$, then for $(\lambda, \nu) = (f^n, f^m)$ periodic we have 
   \[ \frac{f_\lambda}{f_\nu} = \frac{f_{f^n}}{f_{f^m}} = 1 = 1^{d(\lambda) - d(\mu)}.\]
 Thus, the conclusion of Proposition \ref{prop:irred-implies} holds but the hypothesis fails.
\end{example}

We conclude by observing  that the irreducible representation $\pi_{[f^\infty], 1}$ that Carlsen et al.~\cite{CKSS} associate to the maximal tail $\{ v_1, v_2\}$, the element $1 \in \T^k$, and the cofinal path $f^\infty$ is, up to rescaling, our $\Lambda$-projective representation on $\Lambda^\infty \backslash e^\infty$.

 \subsection{Sufficient conditions for irreducibility}
 \label{sec:sufficient}
 
 { In this subsection we show that, although the necessary conditions for irreducibility described in the previous subsection are not necessarily sufficient, for $\Lambda$-projective representations on the infinite path space of a $k$-graph the situation is different. For example, we obtain a converse to Proposition~\ref{ventosul} in Theorem~\ref{thm:ergodic} below (with $E =X = \Lambda^\infty$).} 
 {If $\Lambda$ has no sinks, then we obtain an alternative sufficient condition for irreducibility in Theorem \ref{thm:ergodic-inf-path-space}.  Finally, we identify another sufficient condition in Proposition \ref{thm:OA-irred} by relating the $\Lambda$-projective representation of $C^*(\Lambda)$ on $\Lambda^\infty$ to a representation of a related 1-graph.  A variety of examples indicate that each of these sufficient conditions has a different domain of applicability.}

 \begin{thm} \label{thm:ergodic}
    Let $\Lambda$ be a row-finite $k$-graph with no sources.  
  Suppose that the infinite path space $\Lambda^\infty$ admits a $\Lambda$-projective system on $(\Lambda^\infty,\mu)$, for some {Radon} measure $\mu$ with standard prefixing maps $\{\sigma_\lambda\, :\, \lambda\in\Lambda\}$, coding maps $\{\sigma^n\, :\, n\in\N^k\}$, and functions $\{f_\lambda\,:\,\lambda\in\Lambda\}$ satisfying Conditions (a) and (b) of Definition~\ref{def:lambda-proj-system}, with the measure of all cylinder sets finite.  Let $\{T_\lambda:\lambda\in\Lambda\}$ be the operators given by Equation \eqref{eq:T-lambda} of Proposition \ref{prop:lambda-proj-repn}.  Let $E \subseteq \Lambda^\infty$ satisfy the hypotheses of Proposition \ref{prop:restriction-gives-SBFS}, so that the restriction $\{T_\lambda^E\}_{\lambda \in \Lambda}$ of the $\Lambda$-projective representation to $L^2(E, \mu)$ is again a $\Lambda$-projective representation.
 Then 
 \begin{itemize}
 \item[(a)] The commutant of the operators $\{T_\lambda^E: \lambda \in \Lambda\}$ consists of multiplication
      operators by functions $h$ with $h \circ \sigma^n = h$, $\mu_E$-a.e., for all $n\in \N^k$.  
 \item[(b)] If $\sigma^n$ is ergodic with respect to $\mu_E$ for some $n \in \N^k$, then $\{T_\lambda^E\}_{\lambda \in \Lambda}$ is
      irreducible. 
 \item[(c)] {If $\{\sigma^n\}$ is a jointly ergodic family with respect to $\mu_E$, then $\{T_\lambda^E\}_{\lambda \in \Lambda}$ is irreducible.}

    \end{itemize}
     \end{thm}
      
 \begin{proof}
 First we prove (a). As established in Proposition \ref{prop:irred-implies-ergodic}, if 
  $T \in B(L^2(\Lambda^\infty, \mu_E))$ is an operator in the commutant of the $\Lambda$-projective representation $\{ T^E_\lambda\}_\lambda$, then $T= M_h$ for some $h \in L^\infty(E, \mu)$.  
 We will show that $h \circ \sigma^n = h$ a.e.~for all $n \in \N^k$.  By Proposition~\ref{prop:irred-implies-ergodic}, this completes the proof of (a).  
 
 Since $T$ commutes with $T^E_\lambda$ for all $\lambda \in \Lambda$, for any $f \in L^2(E, \mu)$ we have 
$ T_\lambda T f = T T_\lambda f$, and consequently 
 \[ 
 h\,f_\lambda|_E \,(f \circ \sigma^n) = (h \circ \sigma^n) f_\lambda|_E  (f \circ \sigma^n)\;\;\text{ whenever } d(\lambda)=n,
 \ f \in  L^2(E, \mu).
 \]
Fix $\lambda \in \Lambda^n$ and consider $f = {\chi_{Z(s(\lambda)) \cap E}}$, which is a nonzero element of $L^2(E, \mu)$ by hypothesis.
 The definition of a $\Lambda$-projective system and  the above equation combine  to 
 reveal that $h \circ \sigma^n = h$, $\mu_E$-a.e.~on $Z(\lambda)$. Since $\Lambda^\infty = \bigsqcup_{\lambda \in \Lambda^n} Z(\lambda)$ for any $n \in \N^k$, it follows that $h \circ \sigma^n = h$ a.e.~on $E$.

 For (b),  choose $T = M_h$ in the commutant of the representation $\{ T_\lambda^E\}_\lambda$, so that $h = h \circ \sigma^{n}$ for all $n$.  Ergodicity of one of the coding maps $\sigma^n$ implies, by item (a) of Lemma~\ref{lem:ergodic-invariant}, that $h$ is constant, and so $\{T_\lambda^E\}_\lambda$ is irreducible.
 
 {Finally, (c) follows as (b) above, using item (b) of Lemma~\ref{lem:ergodic-invariant} this time.}
 \end{proof}
 
 \begin{rmk}
 One can take $E = \Lambda^\infty$ in the previous Theorem; in this case,  combining Theorem \ref{thm:ergodic} with Proposition \ref{ventosul}, we have a necessary and sufficient characterization of the irreducibility of  a $\Lambda$-projective representation on $\Lambda^\infty$.
 \end{rmk}

 We obtain another sufficient condition for the irreducibility of the restriction to $L^2(E, \mu)$ of the representation $\{T_\lambda\}_{\lambda \in \Lambda}$ if, instead of requiring that $\mu(E \cap Z(v))$ be nonzero for all vertices $v$, we ask that $E$ satisfy the following definition.  Recall that a $\mu$-measurable set $E$ is invariant with respect to a function $T$ if $\mu(E \Delta T^{-1}(E)) = 0$, and note that if $E$ is invariant with respect to $T$, then $\mu(E \cap T^{-1}(E)) = \mu(E)$.
 
\begin{defn} \label{def:in-min-inv-subsets} Fix a Radon measure $\mu$ on $
	\Lambda^\infty$.  A Borel subset of non-zero measure $A\subseteq \Lambda^\infty$ is \emph{$\mu$-invariant} if  $A$ is invariant with respect to the standard coding maps $\sigma^n$, for all $n \in \N^k$, i.e., $\mu(A\Delta (\sigma^n)^{-1}(A))=0$, $\forall n$. A $\mu$-invariant subset $E\subseteq \Lambda^\infty$ is \emph{minimal $\mu$-invariant}  if  there is no $\mu$-invariant subset $A$ of $E$ with   $\mu(A \Delta E ) \neq 0$.
\end{defn}

      {To obtain our next  results, we restrict our attention to higher rank graphs with {\em no sinks} -- that is, for every $1 \leq i \leq k$, and for every $v \in \Lambda^0$, we have $\Lambda^{e_i}v \not = \emptyset$. If $\Lambda$ has no sinks then, for every $n$, the coding map $\sigma^n: \Lambda^\infty \to \Lambda^\infty$ is surjective.}
      
      \begin{lemma}\label{lem:min-invar-implies}
{Suppose $\Lambda$ is a row-finite, source-free higher-rank graph with no sinks.}
      
     \begin{enumerate} 
     \item[(a)] Let $A\subseteq \Lambda^\infty$, with $\mu(A)>0$. If  $\mu((\sigma^n)^{-1}(A) \Delta A) = 0$ for some $n$,  then $\mu(A \Delta \sigma^n(A)) = 0$. 
     
     \item[(b)] {If $E\subseteq \Lambda^\infty$  is $\mu$-invariant,  $\mu(E)>0$, and $A\subset E$ satisfies $\mu((\sigma^n|_E)^{-1}(A) \Delta A) = 0$ for some $n$,  then $\mu(A \Delta \sigma^n(A)) = 0$. }

     \item[(c)] {If $E$ is minimal $\mu$-invariant then $\{\sigma^n\}$ is a jointly ergodic family in $E$, that is, if $A\subset E$, with $\mu(A \Delta (\sigma^n|_E)^{-1}(A)) = 0$ for all $n$, then either $\mu(A) = 0$ or $\mu(E \backslash A) = 0$.} 
     
      \end{enumerate} 
      \end{lemma}
      {\begin{proof} 
      \begin{enumerate}
 \item[(a)]      
By definition
 \[ A \Delta \sigma^n(A)  = \{ x \in A: x \not= \sigma^n(y) \text{ for all } y \in A\} \cup \{ x \in A^c: x = \sigma^n(y) \text{ for some } y \in A\}\]
 and  $ (\sigma^n)^{-1}(A) \Delta A  = \{ x \in A: \sigma^n(x) \not \in A\} \cup \{ x \in A^c: \sigma^n(x) \in A\}$.  {Assuming that $\Lambda$ has no sinks,} each of the sets in the description of $A \Delta \sigma^n(A)$ above is contained in  the image under $\sigma^n$ of one of the sets making up $(\sigma^n)^{-1}(A) \Delta A. $  To be precise, 
 \[ \sigma^n(\{ x \in A: \sigma^n(x) \not \in A\} ) =  \{ y \in A^c: y =\sigma^n(x) \text{ for some } x \in A\}  \]
 and, since each coding map is surjective by hypothesis,
 \[ \sigma^n(\{ x \in A^c: \sigma^n(x) \in A\}) \supseteq \{ y \in A: y \not= \sigma^n(x) \text{ for any } x \in A\}.\]
 By Lemma~\ref{lem:measure-zero-preserved}, the claim follows.

\item[(b)]  Let $E\subseteq \Lambda^\infty$, with $\mu(E)>0$, be $\mu$-invariant. By item (a), $\mu(E \Delta \sigma^n(E))= 0$ for all $n$. Let $A\subset E$ be such that $\mu((\sigma^n|_E)^{-1}(A) \Delta A) = 0$ for some $n$.
Notice that \[ (\sigma^n|_E)^{-1}(A) \Delta A  = \{ x \in A: \sigma^n(x) \not \in A\} \cup \{ x \in E \setminus A: \sigma^n(x) \in A\},\] and that $A \Delta \sigma^n(A)\subseteq E$ a.e.. Proceeding  as in item (a), we conclude that \[A \Delta \sigma^n(A) \subseteq \sigma^n \left((\sigma^n|_E)^{-1}(A) \Delta A\right) \ \text{ a.e.},\]
 where the inclusion 
 \[ \sigma^n(\{ x \in E\setminus A: \sigma^n(x) \in A\}) \supseteq \{ y \in A: y \not= \sigma^n(x) \text{ for any } x \in A\} \ \text{a.e.}\]
 follows from $A\subseteq E$ and $\mu(E \Delta \sigma^n(E))= 0$.
 The proof now follows as in the proof of item~(a), taking $B := (\sigma^n|_E)^{-1}(A) \Delta A$  and $C : = A \Delta \sigma^n(A) $.

\item[(c)] Let $E$ be a minimal $\mu$-invariant set, and let $A\subset E$ be such that $\mu(A)>0$ and $\mu(A \Delta (\sigma^n|_E) ^{-1}(A)) = 0$ for all $n$. Consequently, $\mu(\{ x \in A: \sigma^n(x) \not \in A\}) = 0$ and $\mu(\{ x \in E \backslash A: \sigma^n(x) \in A\})  = 0$. If $A$ is not $\mu$-invariant, then there exists $n\in \N^k$ such that $\mu(A \Delta (\sigma^n)^{-1}(A)) > 0$, and consequently  $\mu(\{ x \in A^c: \sigma^n(x) \in A\}) > 0$.  As $E$ is $\mu$-invariant,  $\mu(\{ x \in E^c: \sigma^n(x) \in A \subseteq E\})  = 0$.  We conclude that $\mu(\{ x \in E \backslash A : \sigma^n(x) \in A\}) > 0$, which contradicts the fact that $\mu(A \Delta (\sigma^n|_E)^{-1}(A)) = 0$. In other words, $A$ must be $\mu$-invariant. By the minimality of $E$ it follows that $\mu(E\setminus A)=\mu(E \Delta A)=0$, as desired.

\end{enumerate}
      \end{proof}

	{\begin{rmk}
	Lemma \ref{lem:min-invar-implies} need not hold for $k$-graphs with sinks. For a simple example, consider the graph $E$ with vertices $v_1, v_2$ and edges $e_1, e_2$ such that $s(e_1)=r(e_1)=s(e_2)=v_1$ and $r(e_2)=v_2$. Then $\Lambda^\infty = \{e_1^\infty, e_2 e_1^{\infty}\}$. Let $\mu$ be a measure on $\Lambda^\infty$ satisfying $\mu(\{ e^\infty_1\}) > 0$ and $\mu(\{e_2e^\infty_1\}) = 0$. Then none of the conclusions of the above lemma hold.  For example, for any $n \not= 0$,  $\{ e_2 e_1^\infty\} \Delta \sigma^n \{ e_2 e_1^\infty\} = \Lambda^\infty $ has positive measure, while $(\sigma^n)^{-1}(\{e_2 e_1^{\infty}\}) \Delta \{ e_2 e_1^\infty\} = \{e_2 e_1^\infty\}$ has measure zero.
	\end{rmk}
	}
	
 \begin{thm}\label{thm:ergodic-inf-path-space}
   Let {$\Lambda$ be a row-finite, source-free $k$-graph with no sinks, and} {$\mu$ a Radon measure on $\Lambda^\infty$ which gives rise to a $\Lambda$-projective system.  Suppose} $E\subseteq \Lambda^\infty$ is a minimal $\mu$-invariant subset. Then the restriction of $\{T_\lambda:\lambda\in\Lambda\}$ to $L^2(E,\mu)$ gives an irreducible representation of $C^*(\Lambda)$.
\end{thm}

\begin{proof}
We first show that $T_\lambda (L^2(E,\mu))\subseteq L^2(E,\mu)$.  By Lemma~\ref{lem:min-invar-implies},  if $\mu((\sigma^n)^{-1}(E) \Delta E) = 0$ for all $n$, then $\mu(E \Delta \sigma^n(E)) = 0$ for all $n$ as well.   
If we fix $f \in L^2(E, \mu)$ and $\lambda \in \Lambda$, then Condition (a) of Definition \ref{def:lambda-proj-system} implies that 
\begin{align*}
\int_E | T_\lambda(f)|^2 \, d\mu &= \int_E |f_\lambda|^2 \cdot \left| f \circ \sigma^{d(\lambda)}\right |^2 \, d\mu = \int_E \left |f \circ \sigma^{d(\lambda)}\right |^2 \frac{d(\mu \circ (\sigma_\lambda)^{-1})}{d\mu} d\mu \\
&= \int_{E}   \left|f \circ \sigma^{d(\lambda)}\right |^2 d(\mu \circ (\sigma_\lambda)^{-1})  = \int_{(\sigma_\lambda)^{-1}(E)} |f|^2 \, d\mu \leq \int_{\sigma^{d(\lambda)}(E)} |f|^2 \, d\mu \\
& = \int_E |f|^2 \, d\mu < \infty,
\end{align*}
where the second line holds because $(\sigma_\lambda)^{-1}(E) \subseteq \sigma^{d(\lambda)}(E)$.  It follows that  $T_\lambda (L^2(E,\mu))\subseteq L^2(E,\mu)$ as claimed.

{Now, suppose $T $ commutes with $T_\lambda|_{L^2(E, \mu)}$ for all $\lambda \in \Lambda$.  As $T_\lambda|_{L^2(E, \mu)} T_\lambda^*|_{L^2(E, \mu)} =  M_{\chi_{Z(\lambda) \cap E}}$, a careful read of the proof of  Theorem~\ref{thm:ergodic}(a)  reveals that, in this case as well, we can conclude that $T = M_h$ for some $h\in L^\infty(E, \mu)$.  Thus, if $f \in L^2(E, \mu)$ we have 
\begin{equation}
h \cdot f_\lambda \cdot (f \circ \sigma^{d(\lambda)}) = f_\lambda \cdot (h \circ \sigma^{d(\lambda)}) \cdot (f \circ \sigma^{d(\lambda)})
\label{eq:h-vs-h-sigma}
\end{equation}
as functions in $L^2(E, \mu)$. The fact that $f_\lambda$ is nonzero precisely on $Z(\lambda)$ means that both sides of Equation \eqref{eq:h-vs-h-sigma} are supported on $Z(\lambda) \cap E$.

If $\mu(E \cap Z(\lambda)) \not= 0$, the fact that $E$ is invariant with respect to $\sigma^{d(\lambda)}$ implies that $\mu(E \cap Z(\lambda)) = \mu(E \cap Z(s(\lambda)))$, and so $f = \chi_{Z(s(\lambda))}$ is a nonzero element of $L^2(E, \mu)$. Consequently, Equation \eqref{eq:h-vs-h-sigma} implies that $h  = h \circ \sigma^{d(\lambda)},\, \mu_E$-a.e. on $Z(\lambda)$, whenever $\mu_E(Z(\lambda))\not= 0$.

Recall that, for any $n\in \N^k$,  we have $E = \bigsqcup_{\lambda \in \Lambda^n} Z(\lambda) \cap E$.  It now follows that $h = h\circ \sigma^n$ as functions in $L^2(E, \mu)$ for all $n\in \N^k$.  {So, $h$ is jointly invariant with respect to $\{\sigma^n\}$. Since, by Lemma~\ref{lem:min-invar-implies}, the maps $\{\sigma^n\}$ are jointly ergodic in $E$, it follows from Lemma~\ref{lem:ergodic-invariant} that $h$ is constant, 
that is,  $h \in \C \chi_E$.}
Thus the restriction of $\{T_\lambda:\lambda\in\Lambda\}$ to $L^2(E,\mu)$ is indeed an irreducible representation.}
\end{proof}

\begin{rmk} If a monic representation of a sink-free, source-free, row-finite $k$-graph has an atom, then Theorem \ref{thm:ergodic-inf-path-space}
implies that the restriction of the representation to the orbit of this atom is irreducible. We will study representations with atoms in detail in the next section. For non-atomic measures, a minimal invariant set should be seen as an analogue of an orbit of an atom. 
\end{rmk}

{The following example shows that Theorem \ref{thm:ergodic-inf-path-space} may apply when Theorem \ref{thm:ergodic} does not. {Indeed, this example shows that $\Lambda$-semibranching function systems may give rise to irreducible representations even when $\Lambda$ is not cofinal, if we can apply Theorem \ref{thm:ergodic-inf-path-space}.}

\begin{example}
\label{ex:not-cofinal-but-irred}
Consider the following 1-graph $\Lambda$:
\[
\begin{tikzpicture}[scale=1.5]
 \node[inner sep=0.5pt, circle] (v) at (0,0) {$v$};
    \node[inner sep=0.5pt, circle] (w) at (1.5,0) {$w$};
    \node[inner sep=0.5pt, circle] (u) at (1, -0.5) {$u$};
    \draw[-latex, thick] (w) edge [out=50, in=-50, loop, min distance=30, looseness=2.5] (w);
    \draw[-latex, thick] (v) edge [out=130, in=230, loop, min distance=30, looseness=2.5] (v);
\draw[-latex, thick] (w) edge [out=150, in=30] (v);
\draw[-latex, thick] (u) edge [out=180, in=270] (v);
    \draw[-latex, thick] (u) edge [out=-43, in=230, loop, min distance=30, looseness=2.5] (u);
\node at (-0.75, 0) {\color{black} $e$}; 
\node at (0.7, 0.45) {\color{black} $g$};
\node at (2.25, 0) {\color{black} $f$};
\node at (0.3, -0.6) {\color{black} $h$};
\node at (1.5, -0.9) {$k$};
\end{tikzpicture}
\]
Define a measure $\mu$ on $\Lambda^\infty$ by  $\mu(Z(e^n) )= \frac{1}{2^{n+1}}, \mu(Z(e^n g)) = \frac{1}{2^{n+3}} = \mu(Z(e^n h)), \mu(Z(u)) = 
\mu(Z(w)) = 1/4$.  (Recall that for a loop $\lambda$, we denote $\lambda^0 = s(\lambda)$.) One easily checks that the usual prefixing and coding maps $\{\sigma^n, \sigma_\lambda\}$ make $(\Lambda^\infty, \mu)$ a $\Lambda$-semibranching function system. 
Observe that each infinite path in $\Lambda^\infty$ except for $e^\infty$ has nonzero $\mu$-measure.  
Set $E = \{ e^n g f^\infty: n \in \N\} \cup \{ f^\infty\}$.  Since $E \cap Z(u) = \emptyset$,  $E$ does not satisfy the hypotheses of Proposition \ref{prop:restriction-gives-SBFS} (and consequently of Theorem \ref{thm:ergodic}).  However, $E$ is  $\mu$-invariant, because $E = \sigma^{-1}(E)$.  In fact, $E$ is minimal $\mu$-invariant: if $A \subseteq E$ and $\mu(A \Delta E ) = \mu(E \backslash A) \not= 0$, then there is an infinite path $x \in E \backslash A$.  Since each infinite path in $E$ has nontrivial measure, and $E$ is $\mu$-invariant, in order to have $\mu(A \Delta \sigma^{-1}(A)) = 0$ we must have $A = \sigma^{-1}(A)$.
Consequently, $A$ cannot contain any path of the form $\lambda \sigma^n(x)$, for $\lambda$ a finite path in $\Lambda$ and $n \in \N$.  As every path in $E$ is of this form, we conclude that the only proper $\mu$-invariant subset $A$ of $E$ is $A = \emptyset$, so $E$ is minimal $\mu$-invariant.

 Theorem \ref{thm:ergodic-inf-path-space} now implies that any $\Lambda$-projective representation of $C^*(\Lambda)$ on $L^2(\Lambda^\infty, \mu)$ will restrict to an irreducible representation on $L^2(E, \mu)$. 
\end{example}
}

 Our final sufficient condition for the irreducibility of a $\Lambda$-projective representation relies on the link between $k$-graphs and directed graphs which we now describe.

If $\Lambda$ is  a row-finite  $k$-graph, for each $1 \leq i \leq k$,  write $A_i$ for the $\Lambda^0 \times \Lambda^0$ (infinite) matrix with entry $A_i(v, w) = | v \Lambda^{e_i} w|$.  Fix  $J=(j_1,\ldots, j_k)$, with all $j_s\in \N_{> 0}$, and define
$A:=A^J=A_1^{j_1}\ldots A_k^{j_k}$.
Then $A$ can also be interpreted as the matrix of a row-finite 1-graph which we will call $\Lambda_A$.
Note that every finite path $\lambda_A \in \Lambda_A$  can be viewed as a finite path 
${\tilde{\lambda}_A} \in \Lambda$ of degree $|\lambda_A| J$ by using the \lq \lq diagonal construction" (see \cite{FGJKP-spectral-triples} for 
the case $j_r=1, \forall r;$ the  general case is similar). More formally, there is a  functor $\Phi_A$ between the path categories
of the 1-graph $ \Lambda_A$ and the $k$-graph $ \Lambda$
 defined  by 
 \[
 \Phi_A(v_A  ) := {\tilde{v}_A},\quad  \Phi_A(\lambda_A  ) := {\tilde{\lambda}_A}.\quad \forall v_A \in \Lambda_A^0, \lambda_A \in \Lambda_A 
 \]
 
Now recall from  \cite[Lemma 5.1]{FGKPexcursions} (cf.~also \cite[Proposition 2.10]{FGJKP-spectral-triples}) 
that the infinite path space $\Lambda^\infty $  of 
$\Lambda$ is homeomorphic and Borel isomorphic to  
the infinite path space $\Lambda_A^\infty $ of the 
$\Lambda_A$:

\begin{prop}\label{prop:homeo-Borel-structure}
There is a canonical  homeomorphism $\Psi_A$  between $ \Lambda_A^\infty $ and $ \Lambda^\infty $ induced by $\Phi_A$. Moreover, $\Psi_A$ also induces an isomorphism of Borel structures.
\end{prop}

\begin{proof}
 The result follows from the observation that the cylinder sets of degree $nJ$, $n \in \N$, generate the topology in $\Lambda^\infty$.  
This observation was established for row-finite source-free $k$-graphs in the references given above, and one easily checks that the proofs given there work for all source-free  row-finite higher-rank graphs.
\end{proof}

Moreover, since 
the Cuntz-Krieger relations for $C^*(\Lambda_A)$ are satisfied by 
$\{ S_{ \tilde{\lambda}_A}: \lambda_A \in \Lambda_A \}, $  the  universal property of graph $C^*$-algebras, see \cite{BS},  implies that there exists a $*$-homomorphism
$\varphi_A: C^*(\Lambda_A)\to C^*(\Lambda)$ such that $\varphi_A(S_{\lambda_A}) = S_{\tilde \lambda_A}$. 
Now, let 
$
\pi: C^*(\Lambda) \to \mathcal{B}\big( L^2(\Lambda^\infty, \mu_\pi)\big)
$
be a monic representation of  $C^*(\Lambda)$.

\begin{proposition} The composition  $\pi_A:= \pi \circ \varphi_A$
	is a monic representation of $C^*(\Lambda_A)$.
	\end{proposition}

\begin{proof} By Proposition
\ref{prop:homeo-Borel-structure},  we know that $L^2(\Lambda^\infty, \mu_\pi)  \cong L^2(\Lambda_A^\infty, \mu_\pi) $. {Moreover, if $\sigma_\lambda, \sigma^n$ denote the standard coding and prefixing maps for $\Lambda$, and we denote by $\tau_\lambda, \tau^n$ the coding maps for $\Lambda_A$, we have 
\[ \tau^n \circ \Psi_A = \sigma^{n J}, \qquad \tau_\lambda \circ \Psi_A = \sigma_{\tilde \lambda}.\]
 } Therefore, by applying Theorem  \ref{thm-characterization-monic-repres} we  conclude that $\pi_A$  is monic. 
	\end{proof}

In addition, Proposition~\ref{prop:homeo-Borel-structure} implies the following result.

 \begin{prop}   The measure $\mu_{\pi_A}$  on $\Lambda^\infty_A$, obtained by integrating against the monic vector the projection-valued measure of Theorem~\ref{prop-proj-valued-measure-gen-case},
is the same as the measure $\mu_\pi$, when 
we identify
$\Lambda^\infty$ and $\Lambda_A^\infty$ using Proposition~\ref{prop:homeo-Borel-structure}. In particular 
the two coding maps satisfy $\sigma_{A}=\sigma_\Lambda^{(j_1,j_2\ldots,j_k)}$.
\end{prop}

Now recall from Theorem \ref{thm:ergodic} that 
the representation $ \pi $ (resp  $ \pi_{A} $) is
irreducible if, and only if, the coding maps $\sigma^n$, $n \in \N^k$  (resp. $\sigma_A^m$, $m \in \N$) 
are jointly ergodic with respect to the measure $\mu_\pi$ (resp. $\mu_{\pi_A}$);  here,  additionally, $\mu_\pi= \mu_{\pi_A}$.
      We thus obtain the following result.

 \begin{prop}\label{thm:OA-irred} Assume, with hypotheses as above,  that the representation $\pi_A$ is irreducible. 
 Then $\pi$ is irreducible.
 \end{prop}

 \begin{proof}

{By Theorem~\ref{thm:ergodic}, it is enough to check that if $B\subset \Lambda^\infty$ satisfies $\mu_\pi \left( (\sigma^{n})^{-1}(B) \Delta B \right)=0$ for all $n\in \N^k$, then either $\mu_\pi(B)=0$ or $\mu_\pi(\Lambda^\infty \setminus B)=0$.  For this, let  $B\subset \Lambda^\infty \cong \Lambda^\infty_A$,  with $ (\sigma^{n})^{-1}(B)= B$ for all $n\in \N^k$;  then in particular 
$ \mu_{\pi_A}((\sigma_A^{m})^{-1}(B)\Delta B) = 0$ for all $m\in \N$.
Since $\pi_A$ is irreducible, Theorem~\ref{thm:ergodic} now implies that
 $\mu_{\pi_A}(B)=0$ or $\mu_{\pi_A}(\Lambda^\infty_A \setminus B)=0$, which also implies  $\mu_\pi(B)=0$ or $\mu_\pi(\Lambda^\infty \setminus B)=0$. }

\end{proof}

 \subsection{Further results on irreducibility}
 \label{sec:random}

{In this subsection we explore a few more results related with the irreducibility of representations arising from $\Lambda$-semibranching function systems.
 Our first result regards the decomposition of a representation as a direct sum of irreducible representations.} For purely atomic representations this follows by restricting the representation to the orbits, see \cite[Corollary~3.4]{FGJKP-atomic} and \cite[Proposition~4.3]{dutkay-jorgensen-atomic}. So we will focus on the case of measures on $\Lambda^\infty$ without atoms (since measures with atoms induce purely atomic representations, see Theorem~\ref{thm:irred-repres-atom-purely-atomic}). In this more general setting, we will need an extra hypothesis to obtain the desired decomposition (see Definition \ref{def:approx-ergodic} below).  
   The section concludes with Proposition \ref{thm:disjoint-repres-general}, which relates the question of when two $\Lambda$-semibranching representations are disjoint to the supports of  the associated measures.
 \begin{defn}
 \label{def:approx-ergodic}
 {Fix a Radon measure $\mu$ on $\Lambda^\infty$.}
 Let $E\subseteq \Lambda^\infty$ be a $\mu$-invariant Borel subset of non-zero measure. We say that $E$ is \emph{$\varepsilon$-approximately ergodic} if there exists $\varepsilon$ such that if $B\subset E$ is invariant and $0<\mu(B)<\mu(E)$, then $\varepsilon< \mu(B) < \mu(E)-\varepsilon$. 
 \end{defn}
 
 \begin{prop}\label{prop:ex-min-ergodic-sets}
Let  $\Lambda$ be a row-finite, source-free $k$-graph with no sinks, and $\mu$ a Radon measure on $\Lambda^\infty$ which gives rise to a $\Lambda$-projective system.  Suppose that there exists an $\varepsilon>0$ such that every $\mu$-invariant subset of non-zero measure is $\varepsilon$-approximately ergodic and that  $\mu(\Lambda^\infty) < \infty.$  Then the representation $\{T_\lambda:\lambda\in\Lambda\}$ splits as a direct sum of irreducible representations.
 \end{prop}
 \begin{proof}
 
 Our goal in this proof is to show that $\Lambda^\infty = \sqcup E_i$, where each $E_i$ is a minimal $\mu$-invariant set.
  Once we have this, the result follows from $L^2(\Lambda^\infty) = \bigoplus L^2(E_i)$ and Theorem~\ref{thm:ergodic-inf-path-space}.

 If the representation given by $\{T_\lambda:\lambda\in\Lambda\}$ is irreducible, we are done. Suppose it is not. Then, {applying Theorem~\ref{thm:ergodic}(c) to the set $E = \Lambda^\infty$}, we see that  there exists a $\mu-$invariant Borel subset $B_1 \subseteq \Lambda^\infty $ such that $0<\mu(B_1)<\mu(\Lambda^\infty)$. 
 Notice that $ \Lambda^\infty \setminus B_1$ is also $\mu$-invariant and $ 0<\mu(\Lambda^\infty \setminus B_1)<\mu(\Lambda^\infty)$. By hypothesis, both $B_1$ and $\Lambda^\infty \setminus B_1$ are $\varepsilon$-approximately ergodic, so  $\varepsilon< \mu(B_1)<\mu(\Lambda^\infty)-\varepsilon$. If $B_1$ is a minimal $\mu$-invariant set, we let $E_1:=B_1$. If not, there is a $\mu$-invariant Borel subset $B_2$ of $B_1$ such that $0<\mu(B_2)<\mu(B_1)$, $B_1 \setminus B_2$ is $\mu$-invariant and $ 0<\mu(B_1\setminus B_2)<\mu(B_1)$. By hypothesis, both $B_2$ and $B_1\setminus B_2$ are $\varepsilon$-approximately ergodic. If $ B_2$ is minimal $\mu$-invariant we let $E_1:=B_2$. If not, then  $\varepsilon< \mu(B_2)<\mu(B_1)-\varepsilon<\mu(\Lambda^\infty)-2\varepsilon$ and proceed inductively. {As $\mu(\Lambda^\infty)$ is finite, this process must terminate; that is,} one of the $B_n$'s must be a minimal $\mu$-invariant set, hence $E_1$ is defined. 
 
 Now, suppose that $E_1=B_N$. Notice that $\mu(E_1)>\varepsilon$ and, furthermore, $\Lambda^\infty \setminus E_1 = (\Lambda^\infty \setminus B_1 )\cup (B_1 \setminus B_2) \cup \ldots \cup (B_{N-1} \setminus B_N)$, which is $\mu$-invariant and has non-zero measure. Applying the same procedure  described in the paragraph above, starting with $\Lambda^\infty \setminus E_1$ instead of $\Lambda^\infty$, we obtain a set $E_2$ such that $\varepsilon<\mu(E_2)$ and $E_2 \subseteq \Lambda^\infty \setminus E_1$. Proceeding inductively, we define the $E_n$. 
 
 Finally, since $\mu(\Lambda^\infty)<\infty$, $\mu(E_n)>\varepsilon$ for every $n$, and the $E_n$ are disjoint, we obtain the desired decomposition into minimal $\mu$-invariant sets, that is, $\Lambda^\infty = \sqcup E_n$.
 \end{proof}

In \cite[Theorem 3.10]{FGJKP-monic} the authors prove that for finite $k$-graphs, and for representations of $C^*(\Lambda$) arising from $\Lambda$-projective systems, on $\Lambda^\infty$, the task of checking
when two representations are equivalent reduces to a measure-theoretical problem. {To be precise, in the finite setting, two $\Lambda$-semibranching representations which arise from measures $\mu, \mu'$ on $\Lambda^\infty$ are disjoint if, and only if, $\mu, \mu'$ are mutually singular.  }As with Theorem~\ref{thm:ergodic}, it is interesting to know if this result holds for general measure spaces. In fact, as before, the result only holds partially, as we show below. 

\begin{prop} (Cf. \cite[Theorem 2.12]{dutkay-jorgensen-monic} and \cite[Theorem 3.10]{FGJKP-monic} )  \label{thm:disjoint-repres-general}
    Let $\Lambda$ be a row-finite $k$-graph with no sources.  
  Suppose that   $L^2(X, \mu)$ and $L^2(X, \mu')$ are two  $\Lambda$-semibranching function systems 
  as in Definitions \ref{def:SBFS} and \ref{def-lambda-SBFS-1},
    with $\sigma$-finite measures $\mu, \mu'$, identical $D_\lambda, R_\lambda,$ and coding maps $\{\tau^n\, :\, n\in\N^k\}$. 
  Choose nonnegative  functions $f_\lambda, f_\lambda'$ satisfying Definition \ref{def:lambda-proj-system} and   
    let $\{T_\lambda:\lambda\in\Lambda\}$ and $\{T_\lambda':\lambda\in\Lambda\}$ be the  
    associated representations of $C^*(\Lambda)$ as in Proposition~\ref{prop:lambda-proj-repn}. 
If the representations  $\{T_\lambda:\lambda\in\Lambda\}$ and $\{T_\lambda':\lambda\in\Lambda\}$ 
      are disjoint, then the two measures $\mu,$ $\mu'$ are mutually singular.       
      
      {Conversely, if $\mu, \mu'$ are  finite measures which are mutually singular and $X = \Lambda^\infty$,  then the associated $\Lambda$-projective representations are disjoint.}     
  \end{prop}
 
 \begin{proof}
 We first observe that \cite[Proposition 3.6]{FGJKP-monic}, although stated for finite $k$-graphs only, still holds for row-finite $k$-graphs.

Assume that the representations $\{ T_\lambda\}_\lambda, \{T_\lambda'\}_\lambda$, are disjoint. Following \cite[Proposition 3.6]{FGJKP-monic}, decompose $ d\mu' = h^2 d\mu +d\nu$, where $\nu$ is supported on $B$ and $\mu$ is supported on $A = X \backslash B$, the sets $A, B$ are invariant under the prefixing and coding maps, and $h \geq 0$. {Since $A$ is invariant under the prefixing map $\tau^n$ for all $ n$, $L^2(A, \mu')$ is an invariant subspace
 for the representation $\{ T_\lambda' : \lambda \in \Lambda\}$. }
 
 Define the operator $W$ on $L^2(X, \mu')$ by $W(f) = f\cdot h$ if $f \in  L^2(A, \mu')$, and $W(f) = 0$ on the
 orthogonal complement of $L^2(A, \mu') \subseteq L^2(X, \mu')$.  To  check that $W$ is intertwining, we recall from {\cite[Proposition 3.6(d)]{FGJKP-monic}}
  and the non-negativity condition on $\{f_\lambda\}$ and $\{f'_\lambda\}$ that $f_\lambda \cdot (h \circ \tau^{d(\lambda)} = f_\lambda' \cdot h$ on $A$.  We consequently  obtain the following equalities for $f \in L^2(A, \mu')$:
 \[
 T_\lambda W (f )= f_\lambda(h \circ \tau^{d(\lambda)})(f \circ \tau^{d(\lambda)}) = f_\lambda' \, h\, (f \circ \tau^{d(\lambda)}) = W T_\lambda'(f).
 \]
 If $f \in L^2(A, \mu')^\perp$ then $W(f) = T_\lambda'(f) = 0$ and so the above equality holds on $L^2(X, \mu')$.
 Since $W$ intertwines the representations $\{ T_\lambda\}_{\lambda \in \Lambda}, \{ T_{\lambda}'\}_{\lambda \in \Lambda}$ of $C^*(\Lambda)$, we must have $W =0$; hence $h=0$, so $\mu, \mu'$ are mutually singular.
 
 {The proof of the final statement follows verbatim as in \cite[Theorem 3.10]{FGJKP-monic}, so we refrain from presenting it here.}
 \end{proof}

\begin{example} For a simple example that the converse of the above proposition does not hold when $X \not= \Lambda^\infty$, let $E$ be the graph with vertices $v_1, v_2$ and edges $e_1, e_2$ such that $s(e_i)=v_i$, $i=1,2$, $r(e_1)=v_2$ and $r(e_2)=v_1$. As with Example~\ref{seal}, we first define a branching system in the sense of \cite{goncalves-li-royer-SBFS}, which then gives us a semibranching function system. Let $X=\{1,2,3,4\}$ and define $R_{e_1}=D_{v_2}=\{1,2\}$ and $R_{e_2}=D_{v_1}=\{3,4\}$. Let $f_{e_1}:D_{s(e_1)}\rightarrow R_{e_1}$ be given by $f_{e_1}(x)=x-2$ and $f_{e_2}:D_{s(e_2)}\rightarrow R_{e_2}$ be given by $f_{e_2}(x)=x+2$. Now we define two mutually singular measures. Let $\mu_1({2i+1})=1$ and $\mu_1({2i+2})=0$  $i=0,1$, and $\mu_2({2i+1})=0$ and $\mu_2({2i+2})=1$  $i=0,1$. Notice that both measures are non-zero on the sets $D_{v_i}$, $i=1,2$. So, following \cite[Proposition~3.12 and Remark~3.11]{goncalves-li-royer-SBFS} we get two semibranching function systems. Let $\pi_1$ and $\pi_2$ be the induced representation on $\ell_2(X,\mu_1)$ and $\ell_2(X,\mu_2)$ respectively. They are clearly unitary equivalent, but the measures $\mu_1$ and $\mu_2$ are mutually singular.
\end{example}

%---------------------------
\section{On atomic irreducible representations of $C^*(\Lambda)$} 
%***********************************************************************************************
\label{sec:atomic}

The study of  (irreducible) purely atomic representations of $C^*(\Lambda)$ was initiated  in \cite{FGJKP-atomic}. Previously to this, in the algebraic context, irreducible representations of Leavitt path algebras (Chen modules) were studied in \cite{chen}. Many results in \cite{FGJKP-atomic} (restricted to 1-graphs) and \cite{chen} can be put in correspondence, as is usual in the development of the theory of graph algebras. The most natural candidate for the analytical counterpart of a purely algebraic (semi)-branching system is to equip the (semi)-branching system with counting measure (see \cite{CantoGoncalves, GoncalvesRoyer1, GoncalvesRoyer2} for some of the algebraic results regarding representations arising from branching systems). However, it is not immediately obvious that a  representation arising from a $\Lambda$-semibranching function system on a measure space $(X, \mu)$ which has an atom  is purely atomic in the sense of \cite{FGJKP-atomic}. We prove in Theorem \ref{thm:atom-means-purely-atomic} below that this is indeed the case, {if the associated $\Lambda$-projective representation is irreducible.}   Our results  in this section (when restricted to the 1-graph case) therefore clarify and strengthen the parallel between the algebraic and analytical settings.  {In particular, we show in Theorem \ref{vinhosoon} that any irreducible $\Lambda$-projective representation on an atomic measure space must be monic.
Finally, 
we finish the paper with an application of our results in the context of Naimark's problem for $k$-graph $C^*$-algebras. }

We begin our analysis  by recalling the definition of purely atomic  representations.

 \begin{defn}  \label{defatomic}
 \cite[Definition~3.1]{FGJKP-atomic}
 
Let $\Lambda$ be a row-finite $k$-graph with no sources.
A representation $\{ t_\lambda\}_{\lambda \in \Lambda}$ of  $C^*(\Lambda)$ on a Hilbert space ${\mathcal H}$ is called \emph{purely atomic}  if there exists a Borel subset $\Omega\subset\Lambda^{\infty}$ such that the projection valued measure $P$ defined on the Borel sets of $\Lambda^{\infty}$ as in Proposition~\ref{prop-proj-valued-measure-gen-case} satisfies
\begin{enumerate}
 \item[(a)] $P(\Lambda^{\infty}\backslash \Omega)\;\;=\; 0_{\mathcal H},$
 \item[(b)] $P(\{\omega\})\not=0_{\mathcal H}$ for all $\omega\in \Omega$,
 \item[(c)] $\bigoplus_{\omega\in \Omega}P(\{\omega\})=\;\text{Id}_{\mathcal H},$
 \end{enumerate}
 where the sum on the left-hand side of (c) converges in the strong operator topology.
 
 Thus, a representation of $C^*(\Lambda)$ is said to be \emph{purely atomic} if the corresponding projection-valued measure is purely atomic on the Borel $\sigma$-algebra ${\mathcal B}_o(\Lambda^{\infty})$ of $\Lambda^\infty$. \end{defn}

Let $\pi$ be a representation of $C^*(\Lambda)$ arising from a $\Lambda$-semibranching function system
on $(X,\mu)$. In the results below, it will be key to find a ``coding" of $X$ in the path space $\Lambda^\infty$. This will be accomplished by defining a map from a ``large" subset of $X$ to the path space $\Lambda^\infty$. The ideas presented here generalize the constructions in \cite{chen, GoncalvesRoyer} to the $k$-graph setting.

Given a $\Lambda$-semibranching function system on $(X, \mu)$, define $Y \subseteq X$ by discarding a set of measure zero, so that
\[ Y = \bigsqcup_{v\in \Lambda^0} (Y \cap R_v), \qquad  Y \cap R_\lambda \subseteq Y \cap D_{r(\lambda)} \ \forall \ \lambda \in \Lambda,\]and 
	 for all $v \in \Lambda^0$ and all $n \in \N^k$,
	\[ Y \cap R_v = \bigsqcup_{\lambda \in v\Lambda^n} Y \cap R_{\lambda}.\]
	We define a map $\phi$ from $Y$ to $\Lambda^\infty$ as follows.  Pick $y \in Y$; then $y \in R_{\lambda_1}$ for a unique $\lambda_1 \in \Lambda^{(1,\ldots, 1)}$. In fact, for any $n \in \N$, we have $y \in R_{\lambda_n}$ for a unique $\lambda_n \in \Lambda^{(n, \ldots, n)}$, and for each $k < n$ we have a decomposition of $\lambda_n$ as $\lambda_n = \lambda_k \lambda_{n,k}$, where $d(\lambda_{n,k}) = (n-k, n-k, \ldots, n-k)$.  Thus, \cite[Remark 2.2]{KP} implies that the sequence $(\lambda_n)_{n\in \N}$ determines a unique path in $\Lambda^\infty$; this is $\phi(y)$. 
	Notice moreover that  the prefixing and coding maps on $Y$ are taken via $\phi$ to the standard prefixing and coding maps $\sigma^n, \sigma_\lambda$ on $\Lambda^\infty$.

	\begin{rmk} The map $\phi$ defined above does not need to be onto. The same example used in \cite{GoncalvesRoyer} in the algebraic context (for algebras associated to ultragraphs) also applies here. More precisely, let $E$ be the graph with one vertex $u$ and two edges, say $e_1$ and $e_2$. Let $R_{e_1}$ and $R_{e_2}$ be two infinite countable disjoint sets, $X=D_u=R_{e_1}\cup R_{e_2}$ with counting measure, and let $\tau_{e_i}:R_{e_i} \rightarrow D_u$ be any bijection, for $i\in \{1,2\}$. Then $X$ is countable, but $\Lambda^\infty$ is not.  
	\end{rmk}
	
	We can now show that an irreducible representation arising from a $\Lambda$-semibranching function system on $(X, \mu)$, where $\mu$ has  an atom, must be purely atomic. {The orbit of an infinite path was defined in Equation \eqref{eq:orbit}.}
	
	\begin{thm} 
	\label{thm:irred-repres-atom-purely-atomic} Let $\Lambda$ be a row-finite source-free $k$-graph and $\pi$  be an irreducible representation of $C^*(\Lambda)$ arising from a  $\Lambda$-semibranching function system on $(X,\mu)$, where $\mu$ has an atom, say $y\in X$.  Then $\pi$ is  purely atomic and the associated projection valued measure is supported on the orbit of $\phi(y)$.
	\label{thm:atom-means-purely-atomic}
	\end{thm}
	\begin{proof}

Recall that the projection valued-measure $P$ associated to $\pi$ assigns to any Borel set $A \subseteq \Lambda^\infty$ a projection $P(A) \in \mathcal{B}(L^2(X,\mu))$, see Proposition~\ref{prop-proj-valued-measure-gen-case}. Also recall that {by Remark \ref{rmk:Lambda-proj-is-mult}},  on cylinder sets we have
$\ t_\lambda t_\lambda^* = M_{\chi_{R_\lambda}}.$
Let $y \in X$ be an atom for $(X, \mu)$ and let   $ \gamma=\phi(y) $. 
By \cite[Proposition~3.8]{FGJKP-atomic} it is enough to prove that $\{\gamma\}$ is an atom of $P$ -- that is, that $P(\{ \gamma\})\not= 0$.  We will show that $P(\{\gamma\})(\delta_y) \not= 0$.

Observe first that  $P(Z(\gamma_n))(\delta_y)  = t_{\gamma_n} t_{\gamma_n}^*(\delta_y) = {M_{\chi_{R_{\gamma_n}}}(\delta_y )}= \delta_y$, {where $\gamma_n$ are the initial segments of $\gamma$ defined in the construction of $\phi(y)$.} Moreover, 
for any $n$, 

\[ P(Z(r(\gamma)) \backslash Z(\gamma_n)) (\delta_y) = \sum_{\gamma_n \not= \lambda \in r(\gamma)\Lambda^{(n, \ldots, n)}} P(Z(\lambda)) (\delta_y)  = \sum_{\gamma_n \not= \lambda \in r(\gamma)\Lambda^{(n, \ldots, n)} }
	M_{\chi_{R_\lambda}} \delta_y  = 0 .\]
By the regularity 
 of the projection-valued measure $P$, it follows that $P(Z(r(\gamma)) \backslash \{\gamma\})(\delta_y) = 0$.  However, since $P(Z(r(\gamma)))(\delta_y) = \delta_y $, the finite additivity of $P$ implies that 
\[ P(\{ \gamma\}) \delta_y = \delta_y.\]
As $\delta_y \in L^2(Y, \mu)$ is nonzero, we conclude that $\{\gamma\}$ is an atom of $P$.
	\end{proof}

We now state a row-finite version of \cite[Theorem~3.12]{FGJKP-atomic}. Since the proof of this theorem is analogous to the finite case proof given in \cite{FGJKP-atomic} we refrain from presenting it here. We only remark that the forward direction of the proof in \cite[Theorem~3.12]{FGJKP-atomic} uses Theorems 3.13 and 4.2  in \cite{FGJKP-monic}, for which we have row-finite versions, see Theorem~\ref{thm:ergodic} and Theorem~\ref{thm-characterization-monic-repres} respectively. 

\begin{thm} \label{thm:monic-rep-one-dim}
\cite[Theorem~3.12]{FGJKP-atomic}
Let $\Lambda$ be a row-finite $k$-graph with no sources. Let $\{t_\lambda: \lambda \in \Lambda \}$ be a purely atomic representation of $C^*(\Lambda)$ on a separable Hilbert space H. Suppose that $t_\lambda t_\lambda^* \neq 0$ for all $\lambda \in \Lambda$. Then the representation is monic if,
and only if, for every atom $x \in \Lambda^\infty$, $P(\{x\})$ is one-dimensional. Moreover,
in this case the associated measure $\mu$ arising from the monic representation (see Equation~\eqref{eq:monic_measure})
is atomic.
\end{thm} 

Using the above theorem we will show that any representation arising from a $\Lambda$-semibranching function system on $(X,\mu)$, where $\mu$ has an atom, is monic and so if such a representation is irreducible we conclude that it is unitarily equivalent to a representation arising from a $\Lambda$-semibranching function system on the path space with the standard coding and prefixing maps (see Theorem~\ref{thm-characterization-monic-repres}). This is the final step to complete the analytical correspondence with the results in \cite{chen, GoncalvesRoyer} for 1-graphs. 

\begin{thm}\label{vinhosoon} Let $\Lambda$ be a row-finite source-free $k$-graph and $\pi$  be an irreducible representation of $C^*(\Lambda)$ arising from a  $\Lambda$-semibranching function system on $(X,\mu)$, where $\mu$ has an atom. Then $\pi$ is monic.
\end{thm}
\begin{proof}

Suppose that $y$ is an atom for $\mu$. Then, by Theorem~\ref{thm:irred-repres-atom-purely-atomic}, $\pi$ is purely atomic and, furthermore, $\pi$ is supported on the orbit of $\phi(y)$. Since $\pi$ is irreducible, by \cite[Proposition~3.10(c)]{FGJKP-atomic}, we have that $\text{dim[Range}(P(\{\phi(y)\})] = 1$; this is also true for the atoms in the orbit of $\phi(y)$, see  \cite[Proposition~3.3]{FGJKP-atomic}.  Now apply Theorem~\ref{thm:monic-rep-one-dim}.
\end{proof}

We finish the paper with an application of our results in the context of Naimark's problem for $k$-graph $C^*$-algebras. 

\subsection{Naimark's problem for graph C*-algebras}
\label{sec:naimark}

Naimark's problem asks whether a $C^*$-algebra that has only one irreducible $*$-representation up to unitary equivalence is isomorphic to the $C^*$-algebra of compact operators on some (not necessarily separable) Hilbert space, see \cite{Nai51}.
Recently, Suri and Tomforde showed that this problem has a positive answer for AF graph $C^*$-algebras as well as for $C^*$-algebras of graphs in which each vertex emits a countable number of edges, see \cite{ST}. It is therefore interesting to obtain a combinatorial description of graphs for which the associated $C^*$-algebra is AF and has a unique irreducible representation up to unitary equivalence. This is partially done in \cite{ST}. In fact, in  \cite[Proposition~3.5]{ST} the authors show that if $E$ is a directed graph such that $C^*(E)$ is AF and has a unique irreducible representation up to unitary equivalence, then one of two distinct possibilities must occur:\footnote{{In \cite{ST}, the authors use the reverse convention regarding ranges and sources of edges.  This means that the statement that follows has been changed.}} Either
\begin{itemize}
\item[(1)] $E$ has exactly one source and no infinite paths; or
\item[(2)] $E$ has no source and $E$ contains an infinite path $\alpha := e_1 e_2 \ldots$ with {$r^{-1}( r(e_i)) = \{ e_i \}$} for all $i \in \N$. 
\end{itemize}

We will show next that the converse of \cite[Proposition~3.5]{ST}  does not hold, that is, we will build a graph such that the associated $C^*$-algebra is AF and satisfies item (2) above, but such that there are two non-equivalent irreducible representations.

\begin{example}\label{chuva}
Let $E$ be the graph with vertices $\{v_i\}$, $i=0,1,2,\ldots$ and $\{w_i\}$, $i=1,2,\ldots$, and edges $e_i$ and $f_i$ such that $r(e_1)=r(f_1)=v_0$, $s(e_i)=v_i$ and $s(f_i)=w_i$ for $i=1,2,\ldots$, and $r(e_i)=v_{i-1}$, $r(f_i)=w_{i-1}$ for $i=2,3,\ldots$.  \end{example}

\begin{prop}
Let $E$ be the graph of Example~\ref{chuva}. Then  $C^*(E)$ is AF and satisfies satisfies item (2) above, but there exist two non-equivalent irreducible representations of $C^*(E)$.
\end{prop}
\begin{proof}
Since $E$ has no loops, $C^*(E)$ is AF (see \cite{kpr}).  

Let $\pi$ be the $\Lambda$-semibranching representation of $C^*(E)$ associated to $(\Lambda^\infty, \mu)$, with the standard coding and prefixing maps and counting measure $\mu$.
Let $p$ be the path $e_1e_2 e_3\ldots$ and denote by $[p]$ the orbit of $p$. Similarly, let $q$ be the path $f_1 f_2 f_3\ldots$ and denote by $[q]$ the orbit of $q$. By Theorem~\ref{thm:irred-repres-atom-purely-atomic}, both the restriction of $\pi$ to $[p]$, and its restriction to $[q]$, are purely atomic.   Moreover,  \cite[Theorem 3.10(c)]{FGJKP-atomic}  implies that both of these restrictions are irreducible. However, they are disjoint (hence not equivalent) by \cite[Theorem~3.10(a)]{FGJKP-atomic}.
\end{proof}

For  $k$-graphs (with $k>1)$, some (partial) results on  Naimark's problem can be obtained as follows. First of all, note that if $\Lambda$ is a row-finite k-graph 
with $C^*(\Lambda)$ having only one irreducible $*$-representation up to unitary equivalence, then  $C^*(\Lambda)$  must be simple by \cite[Lemma 2.4]{ST} (which also holds in this case).
Moreover, if $C^*(\Lambda)$ is separable (this for example happens if $\Lambda^0$ is finite) then, by Rosenberg's theorem (see \cite[Page 487]{ST} and \cite[Theorem~4]{Rose}),
$C^*(\Lambda)$ must be isomorphic to the compact operators. Furthermore, since  $C^*(\Lambda)$ is AF, by \cite[Corollary 3.8]{evans-sims}
$\Lambda$  does not contain generalized cycles with an entrance. On the other hand, if $\Lambda^0$ is finite and $\Lambda$ contains no cycles then, by \cite[Theorem 5.2]{evans-sims}, $C^*(\Lambda)$ is a matrix algebra (because it is simple). Finally, if  $\Lambda$ is a locally convex k-graph  such that $\Lambda^0$  is finite and
$\Lambda$ contains no cycles, then $\Lambda^0$ contains a unique source, and
$C^*(\Lambda) = M_{\Lambda v}(\C)$, see \cite[Corollary 5.7]{evans-sims}.

\end{document}